\documentclass[10pt]{paper}%
\usepackage[T1]{fontenc}
\usepackage{amsmath,amssymb}
\usepackage{fullpage}
\usepackage{amsthm}
\usepackage{amsmath}
\usepackage{graphicx}
\usepackage{color}
\usepackage{amsfonts}
\usepackage{amssymb}%
\providecommand{\U}[1]{\protect\rule{.1in}{.1in}}
\providecommand{\U}[1]{\protect\rule{.1in}{.1in}}
\providecommand{\U}[1]{\protect\rule{.1in}{.1in}}
\providecommand{\U}[1]{\protect\rule{.1in}{.1in}}
\providecommand{\U}[1]{\protect\rule{.1in}{.1in}}

\newcommand{\uemptyset }{{\boldsymbol{\emptyset}}}
\newtheorem{Th}{Theorem}[section]

\newtheorem{lemma}[Th]{Lemma}
\newtheorem{Cor}[Th]{Corollary}
\newtheorem{Prop}[Th]{Proposition}
\newtheorem{Def-Prop}[Th]{Definition-Proposition}

\theoremstyle{definition}
\newtheorem{Def}[Th]{Definition}
\newtheorem{Exa}[Th]{Example}

\theoremstyle{remark}
\newtheorem{Rem}[Th]{Remark}

{\rmfamily}

\begin{document}

\title{A combinatorial decomposition of higher level Fock spaces}
\author{{N. Jacon }\thanks{Universit\'e de Franche-Comt\'e, email :
njacon@univ-fcomte.fr} and {C. Lecouvey}\thanks{Universit\'e de Tours, email :
cedric.lecouvey@lmpt.univ-tours.fr }}
\maketitle
\date{}

\begin{abstract}
We give a simple characterization of the highest weight vertices in the
crystal graph of the level $l$ Fock spaces. This characterization is based on
the notion of totally periodic symbols viewed as affine analogues of reverse
lattice words classically used in the decomposition of tensor products of
fundamental $\mathfrak{sl}_{n}$-modules. This yields a combinatorial
decomposition of the  Fock spaces in their irreducible components and the
branching law for the restriction of the irreducible highest weight
$\mathfrak{sl}_{\infty}$-modules to $\widehat{\mathfrak{sl}_{e}}$.

\end{abstract}

\section{Introduction}

To any $l$-tuple $\mathbf{s\in}\mathbb{Z}^{l}$ is associated a Fock space
$\mathcal{F}_{\mathbf{s}}$ which is a $\mathbb{C}(q)$-vector space with basis
the set of $l$-partitions (i.e. the set of $l$-tuples of partitions).\ This
level $l$ Fock space was introduced in \cite{jim} in order to construct the
irreducible highest weight representations of the quantum groups
${\mathcal{U}}_{q}^{\prime}{(\widehat{\mathfrak{sl}_{e}})}$ and ${\mathcal{U}%
_{q}(\mathfrak{sl}_{\infty})}$. It provides a natural frame for the
simultaneous study of the representation theories of ${\mathcal{U}}%
_{q}^{\prime}{(\widehat{\mathfrak{sl}_{e}})}$ and ${\mathcal{U}_{q}%
(\mathfrak{sl}_{\infty})}$. It moreover permits to categorify the
representation theory of the Ariki-Koike algebras (some generalizations of the
Hecke algebras of the symmetric groups) in the nonsemisimple case (see
\cite{arikilivre}).

The Fock space $\mathcal{F}_{\mathbf{s}}$ has two structures of ${\mathcal{U}%
}_{q}^{\prime}{(\widehat{\mathfrak{sl}_{e}})}$ and ${\mathcal{U}%
_{q}(\mathfrak{sl}_{\infty})}$-modules.\ For these two structures, the empty
$l$-partition $\uemptyset$ is a highest weight vector with dominant weights
$\Lambda_{\mathbf{s},e}$ and $\Lambda_{\mathbf{s},\infty}$. We denote by
$V_{e}({\mathbf{s}})$ and $V_{\infty}({\mathbf{s}})$ the corresponding highest
weight ${\mathcal{U}}_{q}^{\prime}{(\widehat{\mathfrak{sl}_{e}})}$ and
${\mathcal{U}_{q}(\mathfrak{sl}_{\infty})}$-modules. In fact, any highest
weight irreducible ${\mathcal{U}}_{q}^{\prime}{(\widehat{\mathfrak{sl}_{e}})}$
or ${\mathcal{U}_{q}(\mathfrak{sl}_{\infty})}$-module can be realized in this
way as the irreducible component of a Fock space with highest weight vector
$\uemptyset$.\ It is also known \cite{AJL,jim} that the two modules structures
are compatible. This means that the action of any Chevalley generator for
${\mathcal{U}}_{q}^{\prime}{(\widehat{\mathfrak{sl}_{e}})}$ can be obtained
from the actions of the Chevalley generators for ${\mathcal{U}_{q}%
(\mathfrak{sl}_{\infty})}$. In particular, $V_{\infty}({\mathbf{s}})$ admits
the structure of a ${\mathcal{U}}_{q}^{\prime}{(\widehat{\mathfrak{sl}_{e}})}$-module.

The first purpose of this paper is to give a simple combinatorial description
of the decomposition of $\mathcal{F}_{\mathbf{s}}$ in its irreducible
${\mathcal{U}}_{q}^{\prime}{(\widehat{\mathfrak{sl}_{e}})}$ and ${\mathcal{U}%
_{q}(\mathfrak{sl}_{\infty})}$-components. For the ${\mathcal{U}%
_{q}(\mathfrak{sl}_{\infty})}$-module structure, this problem is very similar
to the decomposition of a tensor product of fundamental ${\mathcal{U}%
_{q}(\mathfrak{sl}_{n})}$-modules into irreducible ones. It is well-known that
this decomposition can be obtained by using the notion of reverse lattice (or
Yamanouchi) words.\ Here our description of the decomposition into irreducible
is based on the notion of totally periodic symbols which can be regarded as
affine analogues of reverse lattice words.

\noindent The Kashiwara crystal associated to the Fock space $\mathcal{F}%
_{\mathbf{s}}$ admits as set of vertices, the set $\mathcal{G}_{\mathbf{s}}$
of all $l$-partitions. According to Kashiwara crystal basis theory, it
suffices to characterize the highest weight vertices in $\mathcal{G}%
_{\mathbf{s}}$ to obtain the decomposition of $\mathcal{F}_{\mathbf{s}}$ into
its irreducible components. We prove in fact that the totally periodic symbols
label the highest weight vertices of $\mathcal{G}_{\mathbf{s}}$. It is also
worth mentioning that, according to recent papers by Gordon-Losev and
Shan-Vasserot \cite{GL,SV}, there should exist a natural labelling of the
finite dimensional irreducible representations of the rational Cherednik
algebras by highest weight vertices of $\mathcal{G}_{\mathbf{s}}$, thus by a
subset of the set of totally periodic symbols.\ Nevertheless the combinatorial
characterization of this subset seems not immediate.

The set $\mathcal{G}_{\mathbf{s}}$ admits two ${\mathcal{U}}_{q}^{\prime
}{(\widehat{\mathfrak{sl}_{e}})}$ and ${\mathcal{U}_{q}(\mathfrak{sl}_{\infty
})}$-crystal structures. In \cite{JL} we established that the ${\mathcal{U}%
}_{q}^{\prime}{(\widehat{\mathfrak{sl}_{e}})}$-structure of graph on
$\mathcal{G}_{\mathbf{s}}$ is in fact a subgraph of the ${\mathcal{U}%
_{q}(\mathfrak{sl}_{\infty})}$-structure.\ This implies that each
$\ {\mathcal{U}_{q}(\mathfrak{sl}_{\infty})}$-connected component decomposes
into ${\mathcal{U}}_{q}^{\prime}{(\widehat{\mathfrak{sl}_{e}})}$-connected
components.\ Each $l$-partition then admits a ${\mathcal{U}}_{q}^{\prime
}{(\widehat{\mathfrak{sl}_{e}})}$ and a ${\mathcal{U}_{q}(\mathfrak{sl}%
_{\infty})}$-weight. In particular, We can consider the decomposition of the
${\mathcal{U}_{q}(\mathfrak{sl}_{\infty})}$-connected component $\mathcal{G}%
_{\mathbf{s},\infty}(\uemptyset)$ with highest weight vertex $\uemptyset$ in
its ${\mathcal{U}}_{q}^{\prime}{(\widehat{\mathfrak{sl}_{e}})}$-connected
components, \ that is the decomposition of the crystal graph of $V_{\infty
}({\mathbf{s}})$ in ${\mathcal{U}}_{q}^{\prime}{(\widehat{\mathfrak{sl}_{e}}%
)}$-crystals.\ We prove that this decomposition gives the branching law for
the restriction of the ${\mathcal{U}_{q}(\mathfrak{sl}_{\infty})}$-module
$V_{\infty}({\mathbf{s}})$ to ${\mathcal{U}}_{q}^{\prime}{(\widehat
{\mathfrak{sl}_{e}})}$. Observe this does not follow immediately from crystal
basis theory since the root system of affine type $A_{e-1}^{(1)}$ is not
parabolic in the root system of type $A_{\infty}$. We also establish that the
number of highest weight ${\mathcal{U}}_{q}^{\prime}{(\widehat{\mathfrak{sl}%
_{e}})}$-vertices in $\mathcal{G}_{\mathbf{s},\infty}(\uemptyset)$ with fixed
${\mathcal{U}_{q}(\mathfrak{sl}_{\infty})}$-weight is counted by some
particular (skew semistandard) tableaux we call totally periodic. These
tableaux can be regarded as affine analogues of the usual semistandard skew
tableaux relevant for computing the branching coefficients associated to the
restriction of the irreducible ${\mathfrak{gl}}_{n}$-modules to
${\mathfrak{gl}}_{m}\oplus{\mathfrak{gl}}_{n-m}$ with $m<n$ some positive integers.

\noindent It also follows that the number of ${\mathcal{U}}_{q}^{\prime
}{(\widehat{\mathfrak{sl}_{e}})}$-highest weight vertices in $\mathcal{G}%
_{\mathbf{s}}$ with fixed ${\mathcal{U}_{q}(\mathfrak{sl}_{\infty})}$-weight
is finite and can be expressed in terms of the Kostka numbers and the number
of totally periodic tableaux of fixed shape and weight.

\bigskip

The paper is organized as follows. In Section 2, we introduce the notion of
symbol of an $l$-partition. Section 3 is devoted to some background on
$\mathcal{F}_{\mathbf{s}}$, its two module structures and the corresponding
crystal bases theory. In Section 4, we show that the two crystal bases on
$\mathcal{F}_{\mathbf{s}}$ for ${\mathcal{U}}_{q}^{\prime}{(\widehat
{\mathfrak{sl}_{e}})}$ and for ${\mathcal{U}_{q}(\mathfrak{sl}_{\infty})}$ are
compatible. This implies that the decomposition of $\mathcal{G}_{\mathbf{s}%
,e}(\uemptyset)$ into its ${\mathcal{U}}_{q}^{\prime}{(\widehat{\mathfrak{sl}%
_{e}})}$-connected components yields the desired branching law. Section 5
characterizes the highest weight vertices in $\mathcal{G}_{\mathbf{s}}$ by
totally periodic symbols.\ Finally in Section 6, we first express the
multiplicities of the irreducible ${\mathcal{U}_{q}(\mathfrak{sl}_{\infty})}%
$-modules appearing in the decomposition of $\mathcal{F}_{\mathbf{s}}$ in
terms of the Kostka numbers. Next, we establish that the branching
coefficients for the restriction of $V_{\infty}({\mathbf{s}})$ to
${\mathcal{U}}_{q}^{\prime}{(\widehat{\mathfrak{sl}_{e}})}$ can be graded by
the ${\mathcal{U}_{q}(\mathfrak{sl}_{\infty})}$-weights and then counted by
totally periodic semistandard tableaux. This gives the decomposition of
$\mathcal{F}_{\mathbf{s}}$ in its irreducible ${\mathcal{U}}_{q}^{\prime
}{(\widehat{\mathfrak{sl}_{e}})}$-components.

\section{Preliminaries on multipartitions and their symbols}

\subsection{Nodes in multipartitions.}

\label{def1} Let $n\in\mathbb{N}$, $l\in\mathbb{Z}^{l}$ and $\mathbf{s}%
=(s_{0},s_{1},\ldots,s_{l-1})\in\mathbb{Z}^{l}$.

\noindent A \textit{partition} $\lambda$ is a sequence $(\lambda_{1}%
,\ldots,\lambda_{r})$ of decreasing non negative integers.

\noindent An \textit{$l$-partition} (or multipartition) ${\boldsymbol{\lambda
}}$ is an $l$-tuple of partitions $(\lambda^{0},\ldots,\lambda^{l-1}) $. We
write ${\boldsymbol{\lambda}}\vdash_{l}n$ when ${\boldsymbol{\lambda}}$ is an
$l$-partition of total rank $n$. The empty $l$-partition (which is the
$l$-tuple of empty partitions) is denoted by $\uemptyset$.

\noindent If ${\boldsymbol{\lambda}}$ is not the empty multipartition, the
\textit{height} of ${\boldsymbol{\lambda}}$ is by definition the minimal non
negative integer $i$ such that there exists $c\in\{0,\ldots,l-1\}$ satisfying
$\lambda_{i}^{c}\neq0$. By convention, the height of $\uemptyset$ is $0$.

\noindent For all ${\boldsymbol{\lambda}}\vdash_{l}n$, we consider its
\textit{Young diagram}:
\[
\lbrack{\boldsymbol{\lambda}}]=\{(a,b,c)\ a\geq1,\ c\in\{0,\ldots,l-1\},1\leq
b\leq\lambda_{a}^{c}\}.
\]
The \textit{nodes} of ${\boldsymbol{\lambda}}$ are usually defined as the
elements of $[{\boldsymbol{\lambda}}]$. However, by slightly abuse the
notation, they will be regarded in the sequel as the elements of the
(infinite) set:
\[
\{(a,b,c)\ a\geq1,\ c\in\{0,\ldots,l-1\},0\leq b\leq\lambda_{a}^{c}\}.
\]
We define the \textit{content} of a node $\gamma=(a,b,c)\in\lbrack
{\boldsymbol{\lambda}}]$ as follows:
\[
\text{\textrm{cont}}(\gamma)=b-a+s_{c},
\]
and the \textit{residue} $\mathrm{res}(\gamma)$ is by definition the content
of $\gamma$ taken modulo $e$. An \textit{$i$-node} is then a node with residue
$i\in\mathbb{Z}/e\mathbb{Z}$. The nodes of \textit{the right rim} of
${\boldsymbol{\lambda}}$ are the nodes $(a,\lambda_{a}^{c},c)$ with
$\lambda_{a}^{c}\neq0$.
We will say that $\gamma$ is an $i$-node of ${\boldsymbol{\lambda}}$ when
$\mathrm{res}(\gamma)\equiv i(\text{mod }e).$ Finally, We say that $\gamma$ is
\textit{removable} when $\gamma=(a,b,c)\in{\boldsymbol{\lambda}}$ and
${\boldsymbol{\lambda}}\backslash\{\gamma\}$ is an $l$-partition. Similarly
$\gamma$ is \textit{addable} when $\gamma=(a,b,c)\notin{\boldsymbol{\lambda}}$
and ${\boldsymbol{\lambda}}\cup\{\gamma\}$ is an $l$-partition.

\subsection{Symbol of a multipartition}

\label{def2} Let ${\boldsymbol{\lambda}}\vdash_{l}n$. Then one can associate
to ${\boldsymbol{\lambda}}$ its \textit{shifted $\mathbf{s}$-symbol} denoted
by $\mathfrak{B}({\boldsymbol{\lambda}},\mathbf{s})$. Our notation slightly
differs from the one used in \cite[\S 5.5.5]{JG} because the symbols we use
here are semi-infinite with possible negative values. Thus, the symbol
$\mathfrak{B}({\boldsymbol{\lambda}},\mathbf{s})$ is the $l$-tuple
\[
(\mathfrak{B}({\boldsymbol{\lambda}},\mathbf{s})^{0},\mathfrak{B}%
({\boldsymbol{\lambda}},\mathbf{s})^{1},\ldots,\mathfrak{B}%
({\boldsymbol{\lambda}},\mathbf{s})^{l-1})
\]
where for each $c\in\{0,1,\ldots,l-1\}$, and $i=1,2,\ldots$, we have
\[
\mathfrak{B}({\boldsymbol{\lambda}},\mathbf{s})_{i}^{c}=\lambda_{i}%
^{c}-i+s_{c}+1
\]
This symbol is usually represented as an $l$-row tableau whose $c$-th row
(counted from bottom) is $\mathfrak{B}({\boldsymbol{\lambda}},\mathbf{s})^{c}$.

\begin{Exa}
With ${\boldsymbol{\lambda}}=(3,2.2.2,2.1)$ and $\mathbf{s}=(1,0,2)$, we
obtain
\[
\mathfrak{B}({\boldsymbol{\lambda}},\mathbf{s})=\left(
\begin{array}
[c]{rrrrrrr}%
\ldots & -3 & -2 & -1 & 0 & 2 & 4\\
\ldots & -3 & 0 & 1 & 2 &  & \\
\ldots & -3 & -2 & -1 & 0 & 4 &
\end{array}
\right)
\]

\end{Exa}

We make the following observations.

\begin{itemize}
\item It is easy to recover the multipartition ${\boldsymbol{\lambda}}$ and
the multicharge $\mathbf{s}$ from the datum of $\mathfrak{B}%
({\boldsymbol{\lambda}},\mathbf{s})$.

\item For all $c\in\{0,\ldots,l-1\},$ let $j_{c}$ be the maximal integer such
that $\mathfrak{B}({\boldsymbol{\lambda}},\mathbf{s})_{j_{c}}^{c}\neq-{j_{c}%
}+s_{c}+1$, if it exists, we set $j_{c}:=0$ otherwise. Then the entries
$\mathfrak{B}({\boldsymbol{\lambda}},\mathbf{s})_{j}^{c}$ of the symbol such
that $0\leq c\leq l-1$ and $j\leq j_{c}$ are bijectively associated with the
nodes $(j,\lambda_{j}^{c},c)$ of the right rim of ${\boldsymbol{\lambda}}$.
\end{itemize}

\subsection{Period in a symbol}

We now introduce the notion of period in a symbol which is crucial for the sequel.

\begin{Def}
\label{def3}Consider a pair $({\boldsymbol{\lambda}},\mathbf{s})$ and its
symbol $\mathfrak{B}({\boldsymbol{\lambda}},\mathbf{s})$. We say that
$({\boldsymbol{\lambda}},\mathbf{s})$ is \textit{$e$-periodic} if there exists
a sequence $(i_{1},c_{1}),(i_{2},c_{2}),\ldots,(i_{e},c_{e})$ in
$\mathbb{N}\times\{0,1,\ldots,l-1\}$ and $k\in\mathbb{Z}$ satisfying
\[
\mathfrak{B}({\boldsymbol{\lambda}},\mathbf{s})_{i_{1}}^{c_{1}}%
=k,\ \mathfrak{B}({\boldsymbol{\lambda}},\mathbf{s})_{i_{2}}^{c_{2}%
}=k-1,\ldots,\mathfrak{B}({\boldsymbol{\lambda}},\mathbf{s})_{i_{e}}^{c_{e}%
}=k-e+1
\]
and such that

\begin{enumerate}
\item $c_{1}\geq c_{2}\geq\ldots\geq c_{e}$,


\item for all $0\leq c\leq l-1$ and $i\in\mathbb{N}$, we have $\mathfrak{B}%
({\boldsymbol{\lambda}},\mathbf{s})_{i}^{c}\leq k$ (i.e. $k$ is the largest
entry of $\mathfrak{B}({\boldsymbol{\lambda}},\mathbf{s})$).

\item given $t\in\{1,\ldots,e\}$ and $(j,d)$ such that $\mathfrak{B}%
({\boldsymbol{\lambda}},\mathbf{s})_{j}^{d}=k-t+1$, we have $c_{t}\leq
d$.(i.e. there is no entry $k-t+1$ in $\mathfrak{B}({\boldsymbol{\lambda}%
},\mathbf{s})$ strictly below than the one corresponding to $(i_{t},c_{t})$)
%
.
\end{enumerate}

The \textit{$e$-period} of $\mathfrak{B}({\boldsymbol{\lambda}},\mathbf{s})$
is the sequence $(i_{1},\lambda_{i_{1}}^{c_{1}},c_{1}),(i_{2},\lambda_{i_{2}%
}^{c_{2}},c_{2}),\ldots,(i_{e},\lambda_{i_{e}}^{c_{e}},c_{e})$ and the
\textit{form} of the $e$-period is the associated sequence $(k,k-1,\ldots
,k-e+1)$ which can be read in the symbol.
\end{Def}

\subsection{Reading of a symbol}

\label{word} An $e$-period can be easily read on the symbol $\mathfrak{B}%
({\boldsymbol{\lambda}},\mathbf{s})$ associated with $({\boldsymbol{\lambda}%
},\mathbf{s})$ as follows. First, consider the truncated symbol $\mathfrak{B}%
^{t}({\boldsymbol{\lambda}},\mathbf{s})$. It is obtained by keeping only in
$\mathfrak{B}({\boldsymbol{\lambda}},\mathbf{s})$ the entries of the symbol of
the form $\mathfrak{B}({\boldsymbol{\lambda}},\mathbf{s})_{j}^{c}$ for
$j=1,\ldots,h_{c}+e$ (where $h_{c}$ denotes the height of $\lambda^{c}$) and
$c=0,1,\ldots,l-1$.

Denote by $w$ the word with letters in $\mathbb{Z}$ obtained by reading the
entries in the rows of $\mathfrak{B}^{t}({\boldsymbol{\lambda}},\mathbf{s})$
from right to left, next from top to bottom. We say that $w$ is the reading of
$\mathfrak{B}^{t}({\boldsymbol{\lambda}},\mathbf{s})$.\ Each letter of $w$
encodes a node in $({\boldsymbol{\lambda}},\mathbf{s})$ (possibly associated
with a part $0$).

When it exists, the $e$-period of $\mathfrak{B}({\boldsymbol{\lambda}%
},\mathbf{s})$ is the sequence of nodes corresponding to the subword $u$ of
$w$ of the form $u=k(k-1)\cdots(k-e+1)$ where $k$ is the largest integer
appearing in $w$ (and thus also in the symbol) and each letter $k-a,a=0,\ldots
,e-1$ in $t$ is the rightmost letter $k-a$ in $w$.

\begin{Exa}
For $\mathbf{s}=(0,-1,1)$ and ${\boldsymbol{\lambda}}=(3,2.2.2,2.1)$, the
symbol
\[
\mathfrak{B}({\boldsymbol{\lambda}},\mathbf{s})=\left(
\begin{array}
[c]{rrrrrrr}%
\ldots & -4 & -3 & -2 & -1 & 1 & 3\\
\ldots & -4 & -1 & 0 & 1 &  & \\
\ldots & -4 & -3 & -2 & -1 & 3 &
\end{array}
\right)
\]
admits no $4$-period. So $({\boldsymbol{\lambda}},\mathbf{s})$ is not $4$-periodic.

For $\mathbf{s}^{\prime}=(-1,-1,1)$ and ${\boldsymbol{\nu}}%
=(3.3.1,4.3.1,4.4.2)$, we have:
\[
\mathfrak{B}({\boldsymbol{\nu}},\mathbf{s}^{\prime})=\left(
\begin{array}
[c]{rrrrrrrr}%
\ldots & -5 & -4 & -3 & -2 & 1 & \mathbf{4} & \mathbf{5}\\
\ldots & -5 & -4 & -2 & 1 & \mathbf{3} &  & \\
\ldots & -5 & -4 & -2 & \mathbf{1} & \mathbf{2} &  &
\end{array}
\right)
\]
Thus ${\boldsymbol{\lambda}}$ admits a $5$-period with form $(5,4,3,2,1)$. The
word associated $w$ described in \S \ref{word} is:
\[
w=\mathbf{54}1\bar{2}\bar{3}\bar{4}\bar{5}\bar{6}\;\mathbf{3}1\bar{2}\bar
{3}\bar{4}\bar{5}\bar{6}\bar{7}\;\mathbf{21}\bar{2}\bar{4}\bar{5}\bar{6}%
\bar{7}\bar{8}%
\]
where we write $\bar{x}$ for $-x$ for any $x\in\mathbb{Z}_{>0}$.\ So
$\mathfrak{B}({\boldsymbol{\nu}},\mathbf{s}^{\prime})$ is $5$-periodic.
\end{Exa}

\begin{Rem}
\label{Rem_empt_e_period}A pair $(\uemptyset,\mathbf{s})$ is always
$e$-periodic with form of the $e$-period $M,M-1,\ldots,M-e+1$ where
$M=\mathrm{max}(\mathbf{s)}$.
\end{Rem}

\subsection{Removing periods in $\mathfrak{B}(\uemptyset,\mathbf{s})$}

For $l\in\mathbb{N}$ and $e\in\mathbb{N}$, we denote
\[
\mathcal{T}_{l,e}=\{\mathbf{t}=(t_{0},\ldots,t_{l-1})\in\mathbb{Z}^{l}\mid
t_{0}\leq\cdots\leq t_{l-1}\text{ and }t_{l-1}-t_{0}\leq e-1\}.\label{def_Tl}%
\]

We now describe an elementary procedure which permits to associate to any
$l$-tuple $\mathbf{s\in}\mathbb{Z}^{l}$ an element $\mathbf{t\in}%
\mathcal{T}_{l,e}$ such that $\mathfrak{B}(\uemptyset,\mathbf{t})$ is obtained
from $\mathfrak{B}(\uemptyset,\mathbf{s})$ by deleting $e$-periods.

If $\mathbf{s}=\mathbf{s}^{(0)}\mathbf{\notin}\mathcal{T}_{l,e}$, we set
$\mathbf{s}^{(1)}=\mathbf{s}^{\prime}$ where $\mathfrak{B}(\uemptyset
,\mathbf{s}^{\prime})$ is obtained from $\mathfrak{B}(\uemptyset,\mathbf{s})$
by deleting its $e$-period. More generally we define $\mathbf{s}^{(p+1)}$ from
$\mathbf{s}^{(p)}\mathbf{\notin}\mathcal{T}_{l,e}$ such that $\mathbf{s}%
^{(p+1)}=(\mathbf{s}^{(p)})^{\prime}.$

\begin{lemma}
\label{Lem_s(p)}For any $\mathbf{s\in}\mathbb{Z}^{l}$, there exists $p\geq0$
such that $\mathbf{s}^{(p)}\in\mathcal{T}_{l,e}$.
\end{lemma}

\begin{proof}
First observe that for any $i=0,\ldots,l-2$ such that $s_{i+1}-s_{i}<0$, we
have $s_{i+1}^{\prime}-s_{i}^{\prime}\geq s_{i+1}-s_{i}$ with equality if and
only if $s_{i}^{\prime}=s_{i}$ and $s_{i+1}^{\prime}=s_{i+1}$.\ For any
$\mathbf{s\in}\mathbb{Z}^{l}$, set%
\[
f(\mathbf{s)=}\sum_{i=0}^{l-2}\mathrm{min}(0,s_{i+1}-s_{i}).
\]
For any $i=0,\ldots,l-2$ with $s_{i+1}-s_{i}<0$, there is an integer $p$ such
that $s_{i}^{(p)}<s_{i}$ (the $i$-th coordinates of the $l$-tuples
$\mathbf{s}^{(p)},p>0$ cannot be left all untouched by the iteration of our
procedure).\ Therefore, for such a $p$, we have $f(\mathbf{s}^{(p)}%
\mathbf{)>}f(\mathbf{s)}$. Since $f(\mathbf{s})\leq 0$ for any $\mathbf{s\in
}\mathbb{Z}^{l}$, we deduce there exists an integer $p_{0}$ such that $f(\mathbf{s}^{p_0})=0$ and thus such that
$s_{i+1}^{(p_{0})}-s_{i}^{(p_{0})}\geq0$ for any $i=0,\ldots,l-2$. We can thus
assume that the coordinates of $\mathbf{s\in}\mathbb{Z}^{l}$ satisfy
$s_{i+1}-s_{i}\geq0$ for any $i=0,\ldots,l-2$. One then easily verifies that
for any $p\geq0$, the coordinates of $\mathbf{s}^{(p)}$ also weakly
increase.\ Observe that for any $i=0,\ldots,l-2$ such that $s_{i+1}-s_{i}\geq
e$, we have $s_{i+1}^{\prime}-s_{i}^{\prime}\leq s_{i+1}-s_{i}$ with equality
if and only if $s_{i}^{\prime}=s_{i}$ and $s_{i+1}^{\prime}=s_{i+1}$.\ Set%
\[
g(\mathbf{s})=\sum_{i=0}^{l-2}\mathrm{min}(0,e-1-(s_{i+1}-s_{i})).
\]
Assume $\mathbf{s}\mathbf{\notin}\mathcal{T}_{l,e}$. Since a pair
$(s_{i},s_{i+1})$ with $s_{i+1}-s_{i}\geq e$ cannot remain untouched by the
iteration of our procedure, there exists an integer $p$ such that
$g(\mathbf{s}^{(p)})>g(\mathbf{s}).\;$So we have an integer $p_{0}$ such that
$g(\mathbf{s}^{(p_{0})})=0$ and since the coordinates of $\mathbf{s}^{(p_{0}%
)}$ weakly increase, one has $\mathbf{s}^{(p_{0})}\in\mathcal{T}_{l,e}$ as desired.
\end{proof}

\begin{Exa}
Consider $\mathbf{s}=(5,3,5,0,1)$ for $e=3$. We obtain
\begin{align*}
\mathbf{s}^{(0)}  &  =(5,3,5,0,1),\mathbf{s}^{(1)}=(2,3,5,0,1),\mathbf{s}%
^{(2)}=(2,2,3,0,1),\\
\mathbf{s}^{(3)}  &  =(0,2,2,0,1),\mathbf{s}^{(4)}=(-1,0,2,0,1)\text{ and
}\mathbf{s}^{(5)}=(-1,-1,0,0,1),
\end{align*}
and we have $\mathbf{s}^{(5)}\in\mathcal{T}_{5,3}$.
\end{Exa}

\section{Module structures on the Fock space}

We now introduce quantum group modules structures on the Fock space of level
$l$ and describe the associated crystal graphs.

\subsection{Roots and weights}

\label{act1} Let $e\in\mathbb{Z}_{>1}\cup\{\infty\}$. Let ${\mathcal{U}}%
_{q}^{\prime}{(\widehat{\mathfrak{sl}_{e}})}$ (resp. ${\mathcal{U}%
_{q}(\mathfrak{sl}_{\infty})}$) be the quantum group of affine type
$A_{e-1}^{(1)}$ (resp. of type $A_{\infty}$). This is an associative
$\mathbb{Q}(q)$-algebra with generators $e_{i},f_{i},t_{i},t_{i}^{-1}$ with
$i=0,...,e-1$ (resp. $i\in\mathbb{Z}).\;$We refer to \cite[chap. 6]{JG} for
the complete description of the relations between these generators since we do
not use them in the sequel. To avoid repetition, we will attach a label $e $
to the notions we define. When $e$ is finite, they are associated with
${\mathcal{U}}_{q}^{\prime}{(\widehat{\mathfrak{sl}_{e}})}$ whereas the case
$e=\infty$ corresponds to ${\mathcal{U}_{q}(\mathfrak{sl}_{\infty})}$.

We write $\Lambda_{i,e},i=0,...,e-1$ for the fundamental weights. The simple
roots are then given by:
\[
\alpha_{i,e}=-\Lambda_{i-1,e}+2\Lambda_{i,e}-\Lambda_{i+1,e}%
\]
for $i=0,\ldots,e-1$.\ As usual the indices are taken modulo $e$. For
$\mathbf{s}\in\mathbb{Z}^{l}$, we also write $\Lambda_{\mathbf{s},e}%
:=\sum_{0\leq c\leq l-1}\Lambda_{s_{c},e}$.

There is an action of the extended affine symmetric group $\widehat{S}_{l}$ on
$\mathbb{Z}^{l}$ (see \cite[\S 5.1]{JL}). This group is generated by the
elements $\sigma_{1},...,\sigma_{l-1}$ and $y_{0},....,y_{l-1}$ together with
the relations
\[
\sigma_{c}\sigma_{c+1}\sigma_{c}=\sigma_{c+1}\sigma_{c}\sigma_{c+1}%
,\quad\sigma_{c}\sigma_{d}=\sigma_{d}\sigma_{c}\text{ for }\left|  c-d\right|
>1,\quad\sigma_{c}^{2}=1,
\]
\[
y_{c}y_{d}=y_{d}y_{c},\quad\sigma_{c}y_{d}=y_{d}\sigma_{c}\text{ for }d\neq
c,c+1,\quad\sigma_{c}y_{c}\sigma_{c}=y_{c+1}.
\]
for relevant indices. Then we obtain a faithful action of $\widehat{S}_{l}$ on
$\mathbb{Z}^{l}$ by setting for any ${\mathbf{s}}=(s_{0},...,s_{l-1}%
)\in\mathbb{Z}^{l}$%
\[
\sigma_{c}({\mathbf{s}})=(s_{0},...,s_{c},s_{c-1},...,s_{l-1})\text{ and
}y_{c}({\mathbf{s}})=(s_{0},...,s_{c-1},s_{c}+e,...,s_{l-1}).
\]
Given $\mathbf{s,s}^{\prime}\in\mathbb{Z}^{l}$, we have $\Lambda
_{\mathbf{s},e}=\Lambda_{\mathbf{s}^{\prime},e}$ if and only if $\mathbf{s}$
and $\mathbf{s^{\prime}}$ are in the same orbit modulo the action of
$\widehat{S}_{l}$. In this case, we denote $\mathbf{s\equiv}_{e}%
\mathbf{s}^{\prime}$. Set%
\begin{align}
\mathcal{V}_{l,e}  &  =\{\mathbf{v}=(v_{0},\ldots,v_{l-1})\in\mathbb{Z}%
^{l}\mid0\leq v_{0}\leq\cdots\leq v_{l-1}\leq e-1\} \label{def_Vl}%
\end{align}
Given any $\mathbf{s}\in\mathbb{Z}^{l}$ there exists a unique $\mathbf{v}$ in
$\mathcal{V}_{l}$ such that $\mathbf{s\equiv}_{e}\mathbf{v}$.

\subsection{Module structures}

\label{act2} We fix $\mathbf{s}\in\mathbb{Z}^{l}$. The \textit{Fock space}
$\mathcal{F}_{\mathbf{s}}$ is the $\mathbb{Q}(q)$-vector space defined as
follows:
\[
\mathcal{F}_{\mathbf{s}}=\bigoplus_{n\in\mathbb{Z}_{\geq0}}\bigoplus
_{{\boldsymbol{\lambda}}\vdash_{l}n}\mathbb{Q}(q){\boldsymbol{\lambda}}.
\]
According to \cite[\S 2.1]{uglov}, there is an action of $\mathcal{U}%
_{q}^{\prime}(\widehat{\mathfrak{sl}}_{e})$ on the Fock space (see
\cite[\S 6.2]{JG}).\ This action depends on $e$ and we will denote by
$\mathcal{F}_{\mathbf{s},e}$ the $\mathcal{U}_{q}^{\prime}(\widehat
{\mathfrak{sl}}_{e})$-module so obtained. In $\mathcal{F}_{\mathbf{s},e}$,
each partition is a weight vector (with respect to a multicharge $\mathbf{s}$)
with weight given by (see \cite[\S 4.2]{uglov})
\[
\mathrm{wt}({\boldsymbol{\lambda}},\mathbf{s})_{e}=\Lambda_{\mathbf{s},e}%
-\sum_{0\leq i\leq e-1}N_{i}({\boldsymbol{\lambda}},\mathbf{s})\alpha_{i,e},
\]
where $N_{i}({\boldsymbol{\lambda}},\mathbf{s})$ denotes the number of
$i$-nodes in ${\boldsymbol{\lambda}}$ (where the residues are computed with
respect to $\mathbf{s}$). For any $e\in\mathbb{Z}_{>1}\cup\{\infty\}$, the
empty multipartition is always a highest weight vector of weight
$\Lambda_{\mathbf{s},e}$. We write $V_{e}(\mathbf{s})$ for the associated
$\mathcal{U}_{q}^{\prime}(\widehat{\mathfrak{sl}}_{e})$-module. We clearly
have $V_{e}(\mathbf{s})\simeq V_{e}(\mathbf{s}^{\prime})$ if and only if
$\mathbf{s\equiv}_{e}\mathbf{s}^{\prime}$.\ 

In general, the modules structures on $\mathcal{F}_{\mathbf{s}}$ are not
compatible when we consider distinct values of $e$.\ Nevertheless, we have the
following proposition stated in \cite[\S 2.1]{AJL}.

\begin{Prop}
\label{prop_compati}Let $e\in\mathbb{N}_{>0}$.

\begin{enumerate}
\item Any $\mathcal{U}_{q}(\mathfrak{sl}_{\infty})$-irreducible component
$V_{\infty}$ of $\mathcal{F}_{\mathbf{s},\infty}$ is stable under the action
of the ${\mathcal{U}}_{q}^{\prime}{(\widehat{\mathfrak{sl}_{e}})}$-Chevalley
generators $e_{i},f_{i},t_{i},i\in\mathbb{Z}/e\mathbb{Z}$.\ Therefore
$V_{\infty}$ has also the structure of a ${\mathcal{U}}_{q}^{\prime}%
{(\widehat{\mathfrak{sl}_{e}})}$-module.

\item In particular, the ${\mathcal{U}}_{q}^{\prime}{(\widehat{\mathfrak{sl}%
_{e}})}$-module $V_{\infty}({\mathbf{s}})$ is endowed with the structure of a
${\mathcal{U}}_{q}^{\prime}{(\widehat{\mathfrak{sl}_{e}})}$-module.\ Moreover
$V_{e}({\mathbf{s}})$ then coincides with the ${\mathcal{U}}_{q}^{\prime
}{(\widehat{\mathfrak{sl}_{e}})}$-irreducible component of $V_{\infty
}({\mathbf{s}})$ with highest weight vector the empty $l$-partition
\textbf{$\uemptyset$}.
\end{enumerate}
\end{Prop}

\begin{Rem}
The algebras $\mathfrak{sl}_{\infty}$ and ${\widehat{\mathfrak{sl}_{e}}}$ can
be realized as algebras of infinite matrices (see \cite{Kac}). Then
${\widehat{\mathfrak{sl}_{e}}}$ is regarded as a subalgebra of $\mathfrak{sl}%
_{\infty}$. In particular, the irreducible $\mathfrak{sl}_{\infty}$-module of
highest weight $\Lambda_{\mathbf{s},\infty}$ admits the structure of a
${\widehat{\mathfrak{sl}_{e}}}$-module by restriction. The highest
${\widehat{\mathfrak{sl}_{e}}}$-weights involved in its decomposition into
irreducible then coincide with those appearing in the decomposition of
$V_{\infty}({\mathbf{s}})$ into its irreducible ${\mathcal{U}}_{q}^{\prime
}{(\widehat{\mathfrak{sl}_{e}})}$-components.
\end{Rem}

\subsection{Crystal bases and crystal graphs}

\label{act3}We now recall some results on the crystal bases of $\mathcal{F}%
_{\mathbf{s},e}$ established in \cite{jim} and \cite{uglov}. Let
$\mathbb{A}(q)$ be the ring of rational functions without pole at $q=0$.\ Set
\begin{align*}
\mathcal{L}  &  :=\bigoplus_{n\geq0}\bigoplus_{{\boldsymbol{\lambda}}%
\vdash_{l}n}\mathbb{A}(q){\boldsymbol{\lambda}}\text{ and}\\
\mathcal{G}  &  :=\{{\boldsymbol{\lambda}}\text{ }(\text{mod }q\mathcal{L}%
)\mid{\boldsymbol{\lambda}}\text{{\ is an }}l\text{-partition}\}.
\end{align*}

\begin{Th}
[Jimbo-Misra-Miwa-Okado, Uglov]\label{Th_BcF}The pair $(\mathcal{L}%
,\mathcal{G})$ is a crystal basis for $\mathcal{F}_{\mathbf{s},e}$ and
$\mathcal{F}_{\mathbf{s},\infty}.$
\end{Th}

Observe that the crystal basis of the Fock space is the same for
$\mathcal{F}_{\mathbf{s},e}$ and $\mathcal{F}_{\mathbf{s},\infty}$.
Nevertheless, the crystal structures $\mathcal{G}_{e,\mathbf{s}}$ and
$\mathcal{G}_{\infty,\mathbf{s}}$ on $\mathcal{G}$ do not coincide for
$\mathcal{F}_{\mathbf{s},e} $ and $\mathcal{F}_{\mathbf{s},\infty}.$ To
describe these crystal structures we begin by defining a total order on the
removable or addable $i$-nodes. Let $\gamma$, $\gamma^{\prime}$ be two
removable or addable $i$-nodes of ${\boldsymbol{\lambda}}$. We set
\[
\gamma\prec_{\mathbf{s}}\gamma^{\prime}\qquad\overset{\text{def}%
}{\Longleftrightarrow}\qquad\left\{
\begin{array}
[c]{ll}%
\mbox{either} & b-a+s_{c}<b^{\prime}-a^{\prime}+s_{c^{\prime}},\\
\mbox{or} & b-a+s_{c}=b^{\prime}-a^{\prime}+s_{c^{\prime}}\text{ and
}c>c^{\prime}.
\end{array}
\right.
\]



Let ${{{\boldsymbol{\lambda}}}}$ be an $l$-partition. We can consider its set
of addable and removable $i$-nodes. Let $w_{i}({\boldsymbol{\lambda}})$ be the
word obtained first by writing the addable and removable $i$-nodes of
${{{\boldsymbol{\lambda}}}}$ in {increasing} order with respect to
$\prec_{\mathbf{s}}$ next by encoding each addable $i$-node by the letter $A$
and each removable $i$-node by the letter $R$.\ Write $\widetilde{w}%
_{i}({\boldsymbol{\lambda}})=A^{p}R^{q}$ for the word derived from $w_{i}$ by
deleting as many subwords of type $RA$ as possible. The word $w_{i}%
({\boldsymbol{\lambda}})$ is called the ${i}${-word} of ${\boldsymbol{\lambda
}}$ and $\widetilde{w}_{i}({\boldsymbol{\lambda}})$ the {reduced $i$-word} of
${\boldsymbol{\lambda}}$. The addable $i$-nodes in $\widetilde{w}%
_{i}({\boldsymbol{\lambda}})$ are called the {normal addable $i$-nodes}. The
removable $i$-nodes in $\widetilde{w}_{i}({\boldsymbol{\lambda}})$ are called
the {normal removable $i$-nodes}. If $p>0,$ let $\gamma$ be the rightmost
addable $i$-node in $\widetilde{w}_{i}$. The node $\gamma$ is called the {good
addable $i$-node}. If $q>0$, the leftmost removable $i$-node in $\widetilde
{w}_{i}$ is called the {good removable $i$-node}. We set
\begin{equation}
\varphi_{i}({\boldsymbol{\lambda})=p}\text{ and }{\varepsilon}_{i}%
({\boldsymbol{\lambda})=q.} \label{phi-epsi}%
\end{equation}
By Kashiwara's crystal basis theory \cite[\S 4.2]{Kas} we have another useful
expression for $\mathrm{wt}({\boldsymbol{\lambda}},\mathbf{s})_{e}$%

\begin{equation}
\mathrm{wt}({\boldsymbol{\lambda}},\mathbf{s})_{e}=\sum_{i\in\mathbb{Z}%
/e\mathbb{Z}}(\varphi_{i}({\boldsymbol{\lambda})-\varepsilon}_{i}%
({\boldsymbol{\lambda}))\Lambda}_{i,e}\text{.} \label{wt_crystal}%
\end{equation}

We denote by $\mathcal{G}_{e,\mathbf{s}}$ the crystal of the Fock space
computed using the Kashiwara operators $\widetilde{e}_{i}$ and $\widetilde
{f}_{i}$. By \cite{jim}, this is the graph with

\begin{itemize}
\item vertices : the $l$-partitions ${\boldsymbol{\lambda}}\vdash_{l}n$ with
$n\in\mathbb{Z}_{\geq0}$

\item arrows: ${\boldsymbol{\lambda}}\overset{i}{\rightarrow}{\boldsymbol{\mu
}}$ that is $\widetilde{e}_{i}{\boldsymbol{\mu}}={\boldsymbol{\lambda}}$ if
and only if ${{{\boldsymbol{\mu}}}}$ is obtained by adding to
${{{\boldsymbol{\lambda}}}}$ a good addable $i$-node, or equivalently,
${\boldsymbol{\lambda}}$ is obtained from ${\boldsymbol{\mu}}$ by removing a
good removable $i$-node.
\end{itemize}

Note that the order induced by $\prec_{\mathbf{s}}$ does not change if we
translate each component of the multicharge by a common multiple of $e$ (nor
does the associated ${\mathcal{U}}_{q}^{\prime}{(\widehat{\mathfrak{sl}_{e}}%
)}$-weight). Thus, if there exists $k\in\mathbb{Z}$ such that $\mathbf{s}%
=(s_{0}^{\prime},s_{1}^{\prime},\ldots,s_{l-1}^{\prime})=(s_{0}+k.e,s_{1}%
+k.e,\ldots,s_{l-1}+k.e)$ then the crystal $\mathcal{G}_{e,\mathbf{s}}$ and
$\mathcal{G}_{e,\mathbf{s}^{\prime}}$ are identical.

The crystal $\mathcal{G}_{e,\mathbf{s}}$ has several connected components.
They are parametrized by its highest weight vertices which are the
$l$-partitions ${\boldsymbol{\lambda}}$ with no good removable node (that is
such that ${\varepsilon}_{i}({\boldsymbol{\lambda})=0}$). Given such an
$l$-partition ${\boldsymbol{\lambda}}$, we denote by $\mathcal{G}%
_{e,\mathbf{s}}({\boldsymbol{\lambda}})$ its associated connected
component.\ One easily verifies that $\mathrm{wt}(\uemptyset
,\mathbf{s})_{e}=\Lambda_{\mathbf{s}(\text{mod }e)}$.\ So the crystal
$\mathcal{G}_{e,\mathbf{s}}(\uemptyset)$ is isomorphic to the abstract crystal
$\mathcal{G}_{e}(\Lambda_{\mathbf{s}(\text{mod }e)})$. In general, for any
highest weight vertex ${\boldsymbol{\lambda}}$, $\mathcal{G}_{e,\mathbf{s}%
}({\boldsymbol{\lambda}})$ is isomorphic to the abstract crystal
$\mathcal{G}_{e}(\mathrm{wt}({\boldsymbol{\lambda}},\mathbf{s})_{e})$. By
setting $\Lambda_{\mathbf{v}(\text{mod }e)}=\mathrm{wt}({\boldsymbol{\lambda}%
},\mathbf{s})_{e},$ we thus obtain a crystal isomorphism $f_{\mathbf{s}%
,\mathbf{v}}^{e,{\boldsymbol{\lambda}}}:\mathcal{G}_{e,\mathbf{s}%
}({\boldsymbol{\lambda}})\rightarrow\mathcal{G}_{e,\mathbf{v}}(\uemptyset).$

\subsection{Crystal graphs and symbols}

\label{wordsy}

Consider $i\in\mathbb{Z}/e\mathbb{Z}$. The reduced $i$-word $\widetilde{w}%
_{i}$ of a multipartition ${\boldsymbol{\lambda}}$ may be easily computed from
its symbol. Let $j_{low}\in\mathbb{Z}$ be the greatest integer such that
$j_{low}\equiv i($mod $e)$ and such that each row of $\mathfrak{B}%
({\boldsymbol{\lambda}},\mathbf{s})$ contains all the integers lowest or equal
to $j_{low}.$ Such an integer exists since the rows of our symbols are
infinite. For any $j\in\mathbb{Z}$ such that $j\equiv i($mod $e)$ and $j\geq
j_{low}$ let $u_{j}$ be the word obtained by reading in the rows of
$\mathfrak{B}({\boldsymbol{\lambda}},\mathbf{s})$ the entries $j$ or $j+1$
from top to bottom and right to left. Write
\[
u_{i}=\prod_{t=0}^{\infty}u_{j_{0}+te}%
\]
for the concatenation of the words $u_{j}.\ $Here all but a finite number of
words $u_{j_{0}+te}$ are empty. We then encode in $u_{i}$ each letter $j$ by
$A$ and each letter $j+1$ by $R$ and delete recursively the factors $RA$.
Write $\widetilde{u}_{i}$ for the resulting word.


\begin{lemma}
\label{lem_util}We have $\widetilde{w}_{i}=\widetilde{u}_{i}$.
\end{lemma}

\begin{proof}
For any $j\equiv i($mod $e)$, write $w_{j}$ for the word obtained by reading
the addable or removable nodes with content $j$ (with respect to $\mathbf{s}$) successively in the partitions $\lambda^{c},c=l-1,\ldots
,0$.\ Observe there is no ambiguity since each partition $\lambda^{c}$
contains at most one node with content $j$ which is addable or removable. By definition
of the order $\prec_{\mathbf{s}}$, we have
\begin{equation}
w_{i}=\prod_{t=0}^{\infty}w_{j_{0}+te} \label{factorw_i}%
\end{equation}
where all but a finite set of the words $w_{i}$ are empty. Now we come back to
the word $u_{j}$. The contribution to the $c$-th row of $\mathfrak{B}%
({\boldsymbol{\lambda}},\mathbf{s})^{c}$ of $u_{j}$ is one of the factors
$(j+1)j,$ $j+1,$ $j$ or $\emptyset$.\ The factors $(j+1)j$ will be encoded
$RA$ so they will disappear during the cancellation process and we can neglect
their contribution.\ Write $u_{j}^{\prime}$ for the word obtained by deleting
in $u_{j}$ the factors $(j+1)j$ corresponding to entries in the same
row.\ There is a bijection between the letters of $u_{j}^{\prime}$ and $w_{j}$
which associates to each letter $j+1$ (resp. $j$) in $u_{j}^{\prime}$
appearing in the row $c$ a node $R$ (resp. $A$) of $w_{j}$. This easily
implies that $\widetilde{u}_{i}=\widetilde{w}_{i}.$
\end{proof}

\section{Compatibility of crystal bases and weight lattices}

\subsection{Crystal basis of the ${\mathcal{U}_{q}(\widehat{\mathfrak{sl}_{e}%
})}$-module $V_{\infty}(\mathbf{s})$}

Consider $e\in\mathbb{Z}_{>1}\cup\{+\infty\}$.\ The general theory of crystal
bases (see \cite{Kas}) permits to define the Kashiwara operators $\tilde
{e}_{i},\tilde{f}_{i},i\in\mathbb{Z}/e\mathbb{Z}$ on the whole Fock space
$\mathcal{F}_{\mathbf{s},e}$ by decomposing, for any $i\in\mathbb{Z}%
/e\mathbb{Z}$, $\mathcal{F}_{\mathbf{s},e}$ in irreducible ${\mathcal{U}}%
_{q}^{\prime}{(\widehat{\mathfrak{sl}_{e}})}_{i}$ components. These operators
do not depend on the decomposition considered (see \cite[\S 4.2]{Kas}). This
implies that the Kashiwara operators associated with any ${\mathcal{U}}%
_{q}^{\prime}{(\widehat{\mathfrak{sl}_{e}})}$-submodule $M_{e}$ of
$\mathcal{F}_{\mathbf{s},e}$ are obtained by restriction of the Kashiwara
operators defined on $\mathcal{F}_{\mathbf{s},e}$.

Set $\mathbf{s}\in\mathbb{Z}^{l}$. By Proposition \ref{prop_compati}, we know
that $V_{\infty}(\mathbf{s})$ has the structure of a ${\mathcal{U}}%
_{q}^{\prime}{(\widehat{\mathfrak{sl}_{e}})}$-module.\ Set $L_{\infty
}(\mathbf{s})=\mathcal{L}\cap V_{\infty}(\mathbf{s})$ and $B_{\infty
}(\mathbf{s})=L_{\infty}(\mathbf{s})/qL_{\infty}(\mathbf{s})$. It immediately
follows from crystal basis theory that the pair $(L_{\infty}(\mathbf{s}%
),B_{\infty}(\mathbf{s}))$ is a crystal basis for $V_{\infty}(\mathbf{s})$
regarded as a ${\mathcal{U}_{q}(\mathfrak{sl}_{\infty})}$-module.\ In fact
this is also true when $V_{\infty}(\mathbf{s})$ is regarded as an
${\mathcal{U}_{q}(\widehat{\mathfrak{sl}_{e}})}$-module.

\begin{Prop}
The pair $(L_{\infty}(\mathbf{s}),B_{\infty}(\mathbf{s}))$ is a ${\mathcal{U}%
_{q}(\widehat{\mathfrak{sl}_{e}})}$-crystal basis of the ${\mathcal{U}%
_{q}(\widehat{\mathfrak{sl}_{e}})}$-module $V_{\infty}(\mathbf{s}).$
\end{Prop}

\begin{proof}
Observe first that we have the weight spaces decompositions%
\[
L_{\infty}(\mathbf{s})=\bigoplus_{\mu\in P_{e}}\mathcal{L}_{\mu}\cap
V_{\infty}(\mathbf{s})\text{ and }B_{\infty}(\mathbf{s})=\bigoplus_{\mu\in
P_{e}}(\mathcal{L}_{\mu}/q\mathcal{L}_{\mu})\cap B_{\infty}(\mathbf{s})
\]
where $P_{e}$ is the weight space of the affine root system of type
$A_{e-1}^{(1)}$. By Theorem \ref{Th_BcF}, for any $i\in\mathbb{Z}/e\mathbb{Z}%
$, $\tilde{e}_{i}$ and $\tilde{f}_{i}$ stabilize $\mathcal{L}$. They also
stabilize the ${\mathcal{U}_{q}(\widehat{\mathfrak{sl}_{e}})}$-submodule
$V_{\infty}(\mathbf{s})$ by the previous discussion. Therefore, they stabilize
$L_{\infty}(\mathbf{s})$ and $B_{\infty}(\mathbf{s})$. Moreover, we have for
any $b_{1},b_{2}\in B_{\infty}(\mathbf{s}),$ $\tilde{f}_{i}(b_{1})=b_{2}$ if
and only if $\tilde{e}_{i}(b_{2})=b_{1}$ since this is true in $\mathcal{B}$.
This shows that the pair $(L_{\infty}(\mathbf{s}),B_{\infty}(\mathbf{s}))$
satisfies the general definition of a crystal basis for the ${\mathcal{U}%
_{q}(\widehat{\mathfrak{sl}_{e}})}$-module $V_{\infty}(\mathbf{s})$.
\end{proof}

Since $(L_{\infty}(\mathbf{s}),B_{\infty}(\mathbf{s}))$ is a crystal basis for
$V_{\infty}(\mathbf{s})$ regarded as a ${\mathcal{U}_{q}(\mathfrak{sl}%
_{\infty})}$-module, $B_{\infty}(\mathbf{s})$ has the structure of a
${\mathcal{U}_{q}(\mathfrak{sl}_{\infty})}$-crystal that we have denoted by
$\mathcal{G}_{\infty,\mathbf{s}}(\uemptyset)$.\ By the previous proposition,
$B_{\infty}(\mathbf{s})$ (which can be regarded as the set of vertices of
$\mathcal{G}_{\infty,\mathbf{s}}(\uemptyset)$) has also the structure of a
${\mathcal{U}}^{\prime}{_{q}(\widehat{\mathfrak{sl}_{e}})}$-crystal that we
denote by $\mathcal{G}_{\infty,\mathbf{s}}^{e}(\uemptyset)$. This crystal is
also a subcrystal of $\mathcal{G}_{e,\mathbf{s}}$ since the actions of the
Kashiwara operators on $\mathcal{G}_{\infty,\mathbf{s}}^{e}(\uemptyset)$ are
obtained by restriction from $\mathcal{G}_{e,\mathbf{s}}$. Let us now recall
the following result obtained in \cite[Theorem 4.2.2]{JL} which shows that
$\mathcal{G}_{e,\mathbf{s}}$ is in fact a subgraph of $\mathcal{G}%
_{\infty,\mathbf{s}}$

\begin{Prop}
Consider ${\boldsymbol{\lambda}}$ and ${\boldsymbol{\mu}}$ two $l$-partitions
such that there is an arrow ${\boldsymbol{\lambda}}\overset{i}{\rightarrow
}{\boldsymbol{\mu}}$ in $\mathcal{G}_{e,\mathbf{s}}$. Let $j\in\mathbb{Z}$ be
the content of the node ${\boldsymbol{\mu}}\setminus{\boldsymbol{\lambda}}$.
Then, we have the arrow ${\boldsymbol{\lambda}}\overset{j}{\rightarrow
}{\boldsymbol{\mu}}$ in $\mathcal{G}_{\infty,\mathbf{s}}$.
\end{Prop}

By combining the two previous propositions, we thus obtain the following corollary.

\begin{Cor}
\label{Cor_branching}The ${\mathcal{U}_{q}(\mathfrak{sl}_{e})}$-crystal
$\mathcal{G}_{\infty,\mathbf{s}}^{e}(\uemptyset)$ is a subgraph of the
${\mathcal{U}_{q}(\mathfrak{sl}_{\infty})}$-crystal $\mathcal{G}%
_{\infty,\mathbf{s}}(\uemptyset)$. It decomposes into ${\mathcal{U}%
_{q}(\widehat{\mathfrak{sl}_{e}})}$-connected components. This decomposition
gives the decomposition of $V_{\infty}(\mathbf{s})$ into its irreducible
${\mathcal{U}_{q}(\widehat{\mathfrak{sl}_{e}})}$-components.
\end{Cor}

\subsection{Weights lattices}

Let $P_{e}$ and $P_{\infty}$ be the weight lattices of ${\mathcal{U}}%
_{q}^{\prime}{(\widehat{\mathfrak{sl}_{e}})}$ and ${\mathcal{U}_{q}%
(\mathfrak{sl}_{\infty})}$. We have a natural projection defined by%
\begin{equation}
\pi:\left\{
\begin{array}
[c]{l}%
P_{\infty}\rightarrow P_{e}\\
\Lambda_{j,\infty}\longmapsto\Lambda_{j\text{ mod}e,e}%
\end{array}
\right.
\end{equation}
Consider $\mathbf{s\in}\mathbb{Z}^{l}$ and ${\boldsymbol{\lambda}}$ an $l$-partition.

\begin{lemma}
We have $\mathrm{wt}({\boldsymbol{\lambda}},\mathbf{s})_{e}=\pi(\mathrm{wt}%
({\boldsymbol{\lambda}},\mathbf{s})_{\infty})$.
\end{lemma}

\begin{proof}
By (\ref{wt_crystal}), for any $e\in\mathbb{Z}_{>1}$, the coordinate of
$\mathrm{wt}({\boldsymbol{\lambda}},\mathbf{s})_{e}$ on $\Lambda_{i,e}$ is
also equal to the number of letters $A$ in $u_{i}$ minus the number of letters
$R$. This is equal to the sum over the integer $j$ such that $j\equiv i(mode)$
of the number of letters $A$ in $u_{j}$ minus the number of letters $R$.\ The
coordinate of $\mathrm{wt}({\boldsymbol{\lambda}},\mathbf{s})_{e}$ on
$\Lambda_{i,e}$ is thus equal to the sum of the coordinates of $\mathrm{wt}%
({\boldsymbol{\lambda}},\mathbf{s})_{\infty}$ on the $\Lambda_{j,\infty}$ with
$j\equiv i(mode)$ as desired.
\end{proof}

\bigskip

One easily verifies that the kernel of $\pi$ is generated by the $\omega
_{k}:=\Lambda_{k+1,\infty}-\Lambda_{k-e,\infty},k\in\mathbb{Z}$.\ The weight
$\omega_{k}$ have level $0$. In fact level $0$ weights for ${\mathcal{U}%
_{q}(\mathfrak{sl}_{\infty})}$ are the $\mathbb{Z}$-linear combinations of the
elementary weights $\varepsilon_{j}=\Lambda_{j+1,\infty}-\Lambda_{j,\infty}$,
$j\in\mathbb{Z}$.\ The contribution of an entry $j\in\mathbb{Z}$ of
$\mathfrak{B}({\boldsymbol{\lambda}},\mathbf{s})$ to the weight $\mathrm{wt}%
({\boldsymbol{\lambda}},\mathbf{s})_{\infty}$ is exactly $\varepsilon_{j}$. We
also have $\omega_{k}=\varepsilon_{k}+\cdots+\varepsilon_{k-e+1}$.

\section{A combinatorial characterization of the highest weight vertices}

Our aim is now to give a combinatorial description of the highest weights
vertices of $\mathcal{G}_{e,\mathbf{s}}$, the crystal of the Fock space
$\mathcal{F}_{e,\mathbf{s}}$.\ Such a vertex is an $l$-partition without good
removable $i$-node for any $i\in\mathbb{Z}/e\mathbb{Z}$.

\subsection{Removing a period in a symbol}

\label{subsec_removeT}

Let ${\boldsymbol{\lambda}}$ be an $l$-partition. We define the $l$-partition
${\boldsymbol{\lambda}}^{-}$ and a multicharge $\mathbf{s}^{-}$ as follows:

\begin{itemize}
\item If ${\boldsymbol{\lambda}}$ is not $e$-periodic then
${\boldsymbol{\lambda}}^{-}:={\boldsymbol{\lambda}}$ and $\mathbf{s}%
^{-}:=\mathbf{s}$.

\item Otherwise, delete the elements of the $e$-period in $\mathfrak{B}%
({\boldsymbol{\lambda}},\mathbf{s})$. This gives a new symbol $\mathfrak{B}%
({\boldsymbol{\mu}},\mathbf{s}^{\prime})$ which is the symbol of an
$l$-partition associated with another multicharge $\mathbf{s}^{\prime}$. We
then set ${\boldsymbol{\lambda}}^{-}:={\boldsymbol{\mu}}$ and $\mathbf{s}%
^{-}:=\mathbf{s}^{\prime}$.
\end{itemize}

\begin{Prop}
\label{epr} Let ${\boldsymbol{\lambda}}$ be an $e$-periodic
multipartition.\ For any $i\in\mathbb{Z}/e\mathbb{Z}$, write $\widetilde
{u}_{i}$ and $\widetilde{u}_{i}^{-}$ for the reduced words obtained from the
symbols ${\boldsymbol{\lambda}}$ and ${\boldsymbol{\lambda}}^{-}$ as in
\S \ref{wordsy}.

\begin{enumerate}
\item $\widetilde{u}_{i}=\widetilde{u}_{i}^{-}$.

\item $\varphi_{i}({\boldsymbol{\lambda}}^{-})=\varphi_{i}%
({\boldsymbol{\lambda}})$ and $\varepsilon_{i}({\boldsymbol{\lambda}}%
^{-})=\varepsilon_{i}({\boldsymbol{\lambda}})$.
\end{enumerate}
\end{Prop}

\begin{proof}
1: Write $(j_{a},\lambda_{j_{a}}^{c_{a}},c_{a}),$ $a=1,\ldots,e$ for the
$e$-period in $\mathfrak{B}({\boldsymbol{\lambda}},\mathbf{s})$. Recall we
have by convention $c_{1}\geq\cdots\geq c_{e}$.\ Consider $i\in\mathbb{Z}%
/e\mathbb{Z}$.\ Let $u_{i}$ be the word constructed in \S \ref{wordsy}. By
definition, there exists a unique $a\in\{1,\ldots,e\}$ such that
$\mathfrak{B}({\boldsymbol{\lambda}},\mathbf{s})_{j_{a}}^{c_{a}}\equiv
i(\text{mod }e)$. Assume first $a>1$. Write $x_{a-1}$ and $x_{a}$ for the
letters of $u_{i}$ associated to $(j_{a},\lambda_{j_{a-1}}^{c_{a-1}},c_{a-1})$
and $(j_{a},\lambda_{j_{a}}^{c_{a}},c_{a})$.\ We have $x_{a-1}=x_{a}+1$.
Set $u_{i}=u_{i}^{\prime}x_{a-1}vx_{a}u_{i}^{\prime\prime}$ where
$u_{i}^{\prime},v,u_{i}^{\prime\prime}$ are words with letters in $\mathbb{Z}%
$.\ By definition of the $e$-period, $v$ is empty or contains only letters
equal to $x_{a}$. Indeed, $x_{a-1}$ should be the rightmost occurrence of the
integer $x_{a-1}$ in $u_{i}$.\ Therefore the contribution of $x_{a-1}$ and
$x_{a}$ can be neglected in the computation of $\widetilde{u}_{i}$ since they
are encoded by symbols $R$ and $A$, respectively. Now assume $a=1$. Write
$y_{1}$ and $y_{e}$ the letters of $u_{i}$ associated with $(j_{1}%
,\lambda_{j_{1}}^{c_{1}},c_{1})$ and $(j_{e},\lambda_{j_{e}}^{c_{e}},c_{e}%
)$.\ We have $y_{e}=y_{1}-e+1$. By definition of $u_{i}$, we can write
$u_{i}=u_{i}^{\prime}y_{e}vy_{1}u_{i}^{\prime\prime}$ where $u_{i}^{\prime
},v,u_{i}^{\prime\prime}$ are words with letters in $\mathbb{Z}$. By
definition of the $e$-period, $v$ is empty or contains only letters $y_{1}$.
Indeed, $y_{e}$ should be the rightmost occurrence of the integer $y_{e}$ in
$u_{i}$.\ Therefore the contribution of $y_{e}$ and $y_{1}$ can be neglected
in the computation of $\widetilde{u}_{i}$ since they are encoded by symbols
$R$ and $A$, respectively. By the previous arguments, we see that the
contribution of the $e$-period in $u_{i}$ can be neglected when we compute
$\widetilde{u}_{i}$. This shows that $\widetilde{u}_{i}=\widetilde{u}%
_{i}^{\prime}$. Assertion 2 follows immediately from 1, (\ref{phi-epsi}) and
Lemma \ref{lem_util}.
\end{proof}

\subsection{The peeling procedure}

\label{subsec-peeling}Given ${\boldsymbol{\lambda}}$ an arbitrary
$l$-partition and $\mathbf{s}$ a multicharge, we define recursively the
$l$-partition ${\boldsymbol{\lambda}}^{\circ}$ and the multicharge
$\mathbf{s}^{\circ}$ as follows:

\begin{itemize}
\item If ${\boldsymbol{\lambda}}$ is not $e$-periodic, or
${\boldsymbol{\lambda}}$ is empty with $\mathbf{s\in}\mathcal{T}_{l,e},$ then
we set ${\boldsymbol{\lambda}}^{\circ}:={\boldsymbol{\lambda}}$ and
$\mathbf{s}^{\circ}:=\mathbf{s}$.

\item Otherwise we set ${\boldsymbol{\lambda}}^{\circ}:={({\boldsymbol{\lambda
}}^{-})}^{\circ}$ and $\mathbf{s}^{\circ}:={(\mathbf{s}^{-})}^{\circ}$.
\end{itemize}

\begin{Rem}
When $\lambda=\uemptyset$, we have $\mathbf{s}^{\circ}:=\mathbf{s}$ only if
$\mathbf{s\in}\mathcal{T}_{l,e}$.
\end{Rem}

\begin{lemma}
\label{Lem_WellDef}The previous procedure terminates, that is the pair
$({\boldsymbol{\lambda}}^{\circ},\mathbf{s}^{\circ})$ is well-defined.
Moreover we have $\mathbf{s}^{\circ}\in\mathcal{T}_{l,e}$ if
${\boldsymbol{\lambda}}^{\circ}=\uemptyset$.
\end{lemma}

\begin{proof}
If ${{\boldsymbol{\lambda}}}$ is not empty and ${{\boldsymbol{\lambda}}%
^{-}\neq{\boldsymbol{\lambda}}}$, then $\left|  {{\boldsymbol{\lambda}}^{-}%
}\right|  <\left|  {{\boldsymbol{\lambda}}}\right|  $.\ So when we apply the
previous procedure to $({\boldsymbol{\lambda},}\mathbf{s})$, we obtain after a
finite number of steps an aperiodic pair $({{\boldsymbol{\lambda}}}^{\prime
}{{,\mathbf{s}}}^{\prime}{{)}}$ or a pair $(\emptyset,{{\mathbf{u}}})$. In the
first case, we have $({{\boldsymbol{\lambda}}}^{\prime}{{,\mathbf{s}}}%
^{\prime}{{)=}}({\boldsymbol{\lambda}}^{\circ},\mathbf{s}^{\circ})$ and the
procedure terminates. In the second case, we have already noticed in Remark
\ref{Rem_HW_infinite} that $(\emptyset,s^{\prime})$ admits an $e$-period. The
lemma then follows from Lemma \ref{Lem_s(p)}.
\end{proof}

\begin{Def}
The pair $\mathfrak{B}({\boldsymbol{\lambda}},\mathbf{s})$ is said to be
\emph{totally periodic} when ${\boldsymbol{\lambda}}^{\circ}=\uemptyset$ and
$\mathbf{s}^{\circ}\in\mathcal{T}_{l,e}$.
\end{Def}

\begin{Exa}
Here are a couple of examples.

\begin{enumerate}
\item First, assume that $e=3$, let $\mathbf{s}=(1,1)$ and let
${\boldsymbol{\lambda}}=(3.3,4.4.3)$. We have
\[
\mathfrak{B}({\boldsymbol{\lambda}},\mathbf{s})=\left(
\begin{array}
[c]{ccccc}%
\ldots & -2 & 2 & 4 & \mathbf{5}\\
\ldots & -2 & -1 & \mathbf{3} & \mathbf{4}%
\end{array}
\right)
\]
If we delete the $3$-period we obtain the symbol:
\[
\mathfrak{B}({\boldsymbol{\mu}},\mathbf{s}^{\prime})=\left(
\begin{array}
[c]{cccc}%
\ldots & -2 & 2 & 4\\
\ldots & -2 & -1 &
\end{array}
\right)
\]
which is the symbol of the bipartition ${\boldsymbol{\mu}}%
={\boldsymbol{\lambda}}^{-}=(\emptyset,4.3)$ with multicharge $\mathbf{s}%
^{-}=(-1,0)$. We don't have any $3$-period so ${\boldsymbol{\lambda}}^{\circ
}=(1,3.2)$ and $\mathbf{s}^{\circ}=(-1,0).$ Note that we have $(-1,0)\equiv
_{e}(0,2)$.

\item Now take $e=4$, let $\mathbf{s}=(4,5)$ and ${\boldsymbol{\lambda}%
}=(2.2.2.1.1,2)$. We obtain the following symbol
\[
\mathfrak{B}({\boldsymbol{\lambda}},\mathbf{s})=\left(
\begin{array}
[c]{cccccccc}%
\ldots & -1 & 0 & 1 & 2 & 3 & 4 & \mathbf{7}\\
\ldots & -1 & 1 & 2 & \mathbf{4} & \mathbf{5} & \mathbf{6} &
\end{array}
\right)
\]
By deleting the $4$-period, we obtain:
\[
\mathfrak{B}({\boldsymbol{\lambda}}^{-},\mathbf{s}^{-})=\left(
\begin{array}
[c]{ccccccc}%
\ldots & -1 & 0 & 1 & 2 & \mathbf{3} & \mathbf{4}\\
\ldots & -1 & \mathbf{1} & \mathbf{2} &  &  &
\end{array}
\right)
\]
Thus, we get ${\boldsymbol{\lambda}}^{-}=(1.1,\emptyset)$ and $\mathbf{s}%
^{-}=(1,4)$. Now deleting the $4$-period, we have:
\[
\mathfrak{B}(({\boldsymbol{\lambda}}^{-})^{-},(\mathbf{s}^{-})^{-})=\left(
\begin{array}
[c]{ccccc}%
\ldots & -1 & \mathbf{0} & \mathbf{1} & \mathbf{2}\\
\ldots & \mathbf{-1} &  &  &
\end{array}
\right)
\]
and we derive $({\boldsymbol{\lambda}}^{-})^{-}=(\emptyset,\emptyset)$ and
$(\mathbf{s}^{-})^{-}=(-1,2)$. Finally, we can delete the $4$-period
$2,1,0,-1$ in the last symbol, this gives%
\[
\mathfrak{B}({\boldsymbol{\lambda}}^{\circ},\mathbf{s}^{\circ})=\left(
\begin{array}
[c]{ccccc}%
\ldots & -1 &  &  & \\
\ldots & -1 &  &  &
\end{array}
\right)
\]
${\boldsymbol{\lambda}}^{\circ}=({\boldsymbol{\lambda}}^{-})^{-}%
=(\emptyset,\emptyset)$ and $\mathbf{s}^{\circ}=(\mathbf{s}^{-})^{-}=(-1,-1)$.
\end{enumerate}
\end{Exa}

\subsection{Crystal properties of periods}

\begin{Prop}
\label{eprprop1} Let $\mathbf{s}\in\mathbb{Z}^{l}$ and let
${\boldsymbol{\lambda}}\vdash_{l}n$. Then for $i\in\{0,1,\ldots,e-1\}$, we
have $\widetilde{e}_{i}({\boldsymbol{\lambda}})=0$ if and only if
$\widetilde{e}_{i}({\boldsymbol{\lambda}}^{-})=0$
\end{Prop}

\begin{proof}
If ${\boldsymbol{\lambda}}^{-}$ is ${\boldsymbol{\lambda}}$ or the empty
$l$-partition, the lemma is immediate. Otherwise it follows from Lemma
\ref{epr}.
\end{proof}

\begin{Prop}
\label{eprprop3} Let ${\boldsymbol{\lambda}}\vdash_{l}n$ be such that
${\boldsymbol{\lambda}}\neq\uemptyset$ and assume that $\widetilde{e}%
_{i}({\boldsymbol{\lambda}})=0$ for any $i\in\mathbb{Z}e/\mathbb{Z}$. Then
${\boldsymbol{\lambda}}$ admits an $e$-period.
\end{Prop}

\begin{proof}
Consider $c_{1}$ minimal such that $\mathfrak{B}({\boldsymbol{\lambda}%
},\mathbf{s})_{1}^{c_{1}}=M$ is the largest entry of $\mathfrak{B}%
({\boldsymbol{\lambda}},\mathbf{s}).$ Let $i\in\mathbb{Z}/e\mathbb{Z}$ be such
that $M\equiv i+1$ $($mod $e)$.\ Then, in the encoding of the letters of
$u_{i}$ by symbols $A$ or $R$, the contribution of $\mathfrak{B}%
({\boldsymbol{\lambda}},\mathbf{s})_{1}^{c_{1}}$ is the rightmost symbol $R$
of $u_{i}$.\ Since $\widetilde{e}_{i}({\boldsymbol{\lambda}})=0$, there exists
in $u_{i}$ an entry $\mathfrak{B}({\boldsymbol{\lambda}},\mathbf{s})_{a_{2}%
}^{c_{2}}$ encoded by $A$ immediately to the right of $\mathfrak{B}%
({\boldsymbol{\lambda}},\mathbf{s})_{1}^{c_{1}}$ (to have a cancellation
$RA$).\ By maximality\ of $M$ and definition of $u_{i}$, we must have
$\mathfrak{B}({\boldsymbol{\lambda}},\mathbf{s})_{a_{2}}^{c_{2}}=M-1$ and
$c_{2}\leq c_{1}$. We can also choose $c_{2}$ minimal such that $\mathfrak{B}%
({\boldsymbol{\lambda}},\mathbf{s})_{a_{2}}^{c_{2}}=M-1$ (or equivalently, the
contribution of $\mathfrak{B}({\boldsymbol{\lambda}},\mathbf{s})_{a_{2}%
}^{c_{2}}$ is the rightmost $A$ in $u_{i}$). Then, the entries in any row $c$
with $c<c_{2}$ are less than $M-1$.\ If we use $\widetilde{e}_{i-1}%
({\boldsymbol{\lambda}})=0$, we obtain similarly an entry $\mathfrak{B}%
({\boldsymbol{\lambda}},\mathbf{s})_{a_{3}}^{c_{3}}$ with $\mathfrak{B}%
({\boldsymbol{\lambda}},\mathbf{s})_{a_{3}}^{c_{3}}=M-2$, $c_{3}\leq c_{2}$
such that the entries in any rows $c$ with $c<c_{3}$ are less than $M-2$. By
induction, this gives a sequence of entries $\mathfrak{B}({\boldsymbol{\lambda
}},\mathbf{s})_{a_{m}}^{c_{m}}=M-m+1$, for $m=1,\ldots,e$, $c_{1}\geq
\cdots\geq c_{e}$ and the entries in any row $c<c_{m}$ are less than $M-m+1$,
that is the desired $e$-period.
\end{proof}

\begin{Prop}
\label{weight} Let $\mathbf{s}\in\mathbb{Z}^{l}$ and let ${\boldsymbol{\lambda
}}\vdash_{l}n$ be such that ${\boldsymbol{\lambda}}\neq\uemptyset$. Assume
that $({\boldsymbol{\lambda}},\mathbf{s})$ admits an $e$-period of the form
$M,M-1,\ldots,M-e+1$. We have

\begin{enumerate}
\item $\mathrm{wt}({\boldsymbol{\lambda}},\mathbf{s})_{e}=\mathrm{wt}%
({\boldsymbol{\lambda}}^{-},\mathbf{s}^{-})_{e}$.

\item $\mathrm{wt}({\boldsymbol{\lambda}},\mathbf{s})_{\infty}=\mathrm{wt}%
({\boldsymbol{\lambda}}^{-},\mathbf{s}^{-})_{\infty}+\omega_{M}$.
\end{enumerate}
\end{Prop}

\begin{proof}
1: Recall that for any $({\boldsymbol{\lambda}},\mathbf{s})$, we have by
\S \ref{act3}
\[
\mathrm{wt}({\boldsymbol{\lambda}},\mathbf{s})=\sum_{i\in\mathbb{Z}%
/e\mathbb{Z}}(\varphi_{i}({\boldsymbol{\lambda})-\varepsilon}_{i}%
({\boldsymbol{\lambda}))\Lambda}_{i}\text{.}%
\]
By assertion 2 of Proposition \ref{epr}, we have $\varepsilon_{i}%
({\boldsymbol{\lambda}})=\varepsilon_{i}({\boldsymbol{\lambda}}^{-})$ and
$\varphi_{i}({\boldsymbol{\lambda}})=\varphi_{i}({\boldsymbol{\lambda}}^{-})$
for all $i\in\mathbb{Z/}e\mathbb{Z}$. Therefore $\mathrm{wt}%
({\boldsymbol{\lambda}},\mathbf{s})_{e}=\mathrm{wt}({\boldsymbol{\lambda}}%
^{-},\mathbf{s}^{-})_{e}$.
%
Assertion 2 follows from the fact that the contribution to each entry
$j\in\mathbb{Z}$ in $\mathfrak{B}({\boldsymbol{\lambda}},\mathbf{s})$ to
$\mathrm{wt}({\boldsymbol{\lambda}},\mathbf{s})_{\infty}$ is $\varepsilon_{j}%
$. So
\[
\mathrm{wt}({\boldsymbol{\lambda}},\mathbf{s})_{\infty}=\mathrm{wt}%
({\boldsymbol{\lambda}}^{-},\mathbf{s}^{-})_{\infty}+\varepsilon_{M}%
+\cdots\varepsilon_{M-e+1}=\mathrm{wt}({\boldsymbol{\lambda}}^{-}%
,\mathbf{s}^{-})_{\infty}+\omega_{M}.
\]
\end{proof}

\subsection{A combinatorial description of the highest weight vertices}

\begin{Th}
\label{Th_HW}Let $\mathbf{s}\in\mathbb{Z}^{l}$ and let ${\boldsymbol{\lambda}%
}\vdash_{l}n$ then $({\boldsymbol{\lambda},\mathbf{s})}$ is a ${\mathcal{U}%
}_{q}^{\prime}{(\widehat{\mathfrak{sl}_{e}})}$-highest weight vertex if and
only if it is totally periodic.
\end{Th}

\begin{proof}
First assume that $({\boldsymbol{\lambda},\mathbf{s})}$ is totally periodic,
that is ${\boldsymbol{\lambda}}^{\circ}$ is the empty $l$-partition and
$\mathbf{s}^{\circ}\in\mathcal{T}_{l,e}$. An easy induction and Proposition
\ref{eprprop1} show that $({\boldsymbol{\lambda}},\mathbf{s})$ is a
${\mathcal{U}}^{\prime}{_{q}(\widehat{\mathfrak{sl}_{e}})}$-highest weight
vertex. In addition, the weight of $({\boldsymbol{\lambda}},\mathbf{s})$ is
equal to the weight of $({\boldsymbol{\lambda}}^{\circ},\mathbf{s}^{\circ})$
by Prop \ref{weight}. Conversely, if ${\boldsymbol{\lambda}}$ is a
${\mathcal{U}_q}^{\prime}(\widehat{\mathfrak{sl}_{e}})$-highest weight
vertex, we know by Prop \ref{eprprop3} that it admits a period and by
Proposition \ref{eprprop1} that ${\boldsymbol{\lambda}}^{-}$ is also a highest
weight vertex. Moreover, for any $\mathbf{s}\notin\mathcal{T}_{l,e}$, we have
seen in Remark \ref{Rem_empt_e_period} that $\mathfrak{B}(\emptyset
,\mathbf{s}^{\circ})$ contains an $e$-period. By Lemma \ref{Lem_WellDef}, this
implies that ${\boldsymbol{\lambda}}^{\circ}$ is empty with $\mathbf{s}%
^{\circ}\in\mathcal{T}_{l,e}$.
\end{proof}

\begin{Rem}
\ \label{Rem_HW_infinite}

\begin{enumerate}
\item We can obtain the highest weight vertices of $\mathcal{G}_{\mathbf{s}%
,\infty}$ by adapting the previous theorem.\ It suffices to interpret
$\mathcal{G}_{\mathbf{s},\infty}$ as the limit when $e$ tends to infinity of
the crystals $\mathcal{G}_{\mathbf{s},e}$.\ Then $({\boldsymbol{\lambda
},\mathbf{s})}$ is a highest weight vertex if and only if $\mathfrak{B}%
({\boldsymbol{\lambda}},\mathbf{s})$ is totally periodic for $e=\infty$.\ A
period for $e=\infty$ is defined as the natural limit of an $e$-period when
$e$ tends to infinity. This is an infinite sequence of the form
$M,M-1,M-2,\ldots$ in $\mathfrak{B}({\boldsymbol{\lambda}},\mathbf{s})$ where
$M$ is the maximal entry of $\mathfrak{B}({\boldsymbol{\lambda}},\mathbf{s})$.
We say that $\mathfrak{B}({\boldsymbol{\lambda}},\mathbf{s})$ is totally
periodic for $e=\infty$ when it reduces to the empty symbol after deletion of
its periods following the procedure described in \S \ \ref{subsec-peeling}. In
this case, since these periods are infinite, a row of the symbol disappears at
each deletion of a period. In particular, there are $l$ infinite periods.

\item Recall that a word $w$ with letters in $\mathbb{Z}$ is a reverse lattice
(or Yamanouchi) word if it can be decomposed into subwords of the form
$a(a-1)\cdots\mathrm{min}(w)$ where $\mathrm{min}(w)$ is the minimal letter of
$w$.\ Let $m$ be the maximal integer in $\mathfrak{B}({\boldsymbol{\lambda}%
},\mathbf{s})$ such that each row of $\mathfrak{B}({\boldsymbol{\lambda}%
},\mathbf{s})$ contains all the integer $k<m$.\ One easily verify that the
periodicity of $\mathfrak{B}({\boldsymbol{\lambda}},\mathbf{s})$ for
$e=\infty$ is equivalent to say that the word $w$ obtained by reading
successively the entries greater or equal to $m$ in the rows of $\mathfrak{B}%
({\boldsymbol{\lambda}},\mathbf{s})$ from left to right and top to bottom is a
reverse lattice word. Indeed, we always dispose in the symbol $\mathfrak{B}%
({\boldsymbol{\lambda}},\mathbf{s})$ of integers less than $m$ to complete any
decreasing sequence $a,a-1,\ldots,m$ into an infinite sequence. Observe that
this imposes in particular that $M\leq\sum_{c=0}^{l-1}(s_{c}-m+1)$.\ We will
see in \S \ \ref{subsec_dec_Fock_infinite} that this easily gives the
decomposition of $\mathcal{F}_{\mathbf{s},\infty}$into its ${\mathcal{U}%
_{q}(\mathfrak{sl}_{\infty})}$-irreducible components.
\end{enumerate}
\end{Rem}

\begin{Exa}
Take $e=4$, $l=3$, $\mathbf{s}=(3,4,6)$ and ${\boldsymbol{\lambda}}%
=(\emptyset,2.2,2.2.1.1.1.1)$. Consider the symbol
\[
\mathfrak{B}({\boldsymbol{\lambda}},\mathbf{s})=\left(
\begin{array}
[c]{cccccccccc}%
\ldots & -2 & -1 & 0 & 2 & 3 & 4 & 5 & \mathbf{7} & \mathbf{8}\\
\ldots & -2 & -1 & 0 & 1 & 2 & \mathbf{5} & \mathbf{6} &  & \\
\ldots & -2 & -1 & 0 & 1 & 2 & 3 &  &  &
\end{array}
\right)  .
\]
By deleting successively the $4$-periods (pictured in bold), we obtain%
\[
\mathfrak{B}({\boldsymbol{\lambda}}^{-},\mathbf{s}^{-})=\left(
\begin{array}
[c]{cccccccc}%
\ldots & -2 & -1 & 0 & 2 & 3 & \mathbf{4} & \mathbf{5}\\
\ldots & -2 & -1 & 0 & 1 & 2 &  & \\
\ldots & -2 & -1 & 0 & 1 & \mathbf{2} & \mathbf{3} &
\end{array}
\right)
\]%
\[
\left(
\begin{array}
[c]{cccccc}%
\ldots & -2 & -1 & 0 & 2 & \mathbf{3}\\
\ldots & -2 & -1 & 0 & 1 & \mathbf{2}\\
\ldots & -2 & -1 & \mathbf{0} & \mathbf{1} &
\end{array}
\right)  ,\left(
\begin{array}
[c]{ccccc}%
\ldots & -2 & -1 & 0 & \mathbf{2}\\
\ldots & -2 & -1 & \mathbf{0} & \mathbf{1}\\
\ldots & -2 & \mathbf{-1} &  &
\end{array}
\right)  ,\left(
\begin{array}
[c]{cccc}%
\ldots & -2 & -1 & 0\\
\ldots & -2 & -1 & \\
\ldots & -2 &  &
\end{array}
\right)  .
\]
Finally we obtain the empty $3$-partition and $\mathbf{s}^{\circ}%
=(-2,-1-0)\in\mathcal{T}_{3,4}$. So $( \boldsymbol{\lambda}, \mathbf{s})$ is a
highest weight vertex.
\end{Exa}

\section{Decomposition of the Fock space}

Consider $\mathbf{s}=(s_{0},\ldots,s_{l})\in\mathbb{Z}^{l}$. We can assume
without loss of generality that $\mathbf{s\in}\mathcal{T}_{l,\infty}$, that is
$s_{0}\leq\cdots\leq s_{l-1}$. The aim of this section is to provide the
decomposition of $\mathcal{G}_{\mathbf{s},e}$ into its connected
${\mathcal{U}}_{q}^{\prime}{(\widehat{\mathfrak{sl}_{e}})}$-components.\ The
multiplicity of an irreducible module in $\mathcal{F}_{\mathbf{s},e}$ can be
infinite. Nevertheless, we have a filtration of the highest weight vertices in
$\mathcal{G}_{\mathbf{s},\infty}$ by their ${\mathcal{U}_{q}(\mathfrak{sl}%
_{\infty})}$-weights.\ We are going to see that the number of totally periodic
symbols of fixed ${\mathcal{U}_{q}(\mathfrak{sl}_{\infty})}$-weight is finite
and can be counted by simple combinatorial objects. We proceed in two steps.
First, we give the decomposition of $\mathcal{G}_{\mathbf{s},\infty}$ into its
${\mathcal{U}_{q}(\mathfrak{sl}_{\infty})}$-connected components, next we give
the decomposition of each crystal $\mathcal{G}_{\infty}(\mathbf{v}%
),\mathbf{v\in}\mathcal{T}_{l,\infty}$ into its ${\mathcal{U}}_{q}^{\prime
}{(\widehat{\mathfrak{sl}_{e}})}$-connected components.

\subsection{Totally periodic tableaux}

Let $\mathbf{t\in}\mathcal{T}_{l,e}$ such that $t_{i}\leq s_{i}$ for any
$c=0,\ldots,l-1$.\ We denote by $\mathbf{s}\setminus\mathbf{t}$ the skew Young
diagram with rows of length $s_{c}-t_{c},c=0,\ldots,l-1$. By a \textit{skew
(semistandard) tableau} of shape $\mathbf{s}\setminus\mathbf{t}$, we mean a
filling $\tau$ of $\mathbf{s}\setminus\mathbf{t}$ by integers such that the
rows of $\tau$ strictly increase from left to right and its column weakly
increase from top to bottom. The weight of $\tau$ is the ${\mathcal{U}%
_{q}(\mathfrak{sl}_{\infty})}$-weight
\[
\mathrm{wt}(\tau)_{\infty}=\sum_{b\in\tau}\varepsilon_{c(b)}%
\]
of level $0$. Here $b$ runs over the boxes of $\mathbf{s}\setminus\mathbf{t}$
and $c(b)$ is the entry of the box $b$ in $\tau$. The trivial tableau of shape
$\mathbf{s}\setminus\mathbf{t}$ denoted $\tau_{\mathbf{s}\setminus\mathbf{t}}$
is the one in which the $c$-th row contains exactly the letters $t_{c}%
+1,\ldots,s_{c}$.

\noindent A tableau is a skew tableau of shape $\mathbf{s}\setminus\mathbf{t}$
where $\mathbf{t}$ is such that $t_{0}=\cdots=t_{l-1}$. In that case
$\lambda=\mathbf{s}\setminus\mathbf{t}$ is an ordinary Young diagram. Given a
level $0$ weight $\mu=\sum_{j\in\mathbb{Z}}\mu_{j}\varepsilon_{j}$ (where all
but a finite number of $\mu_{j}$ are equal to zero), we then denote by
$K_{\lambda,\mu}$ the Kostka number associated to $\lambda$ and $\mu$. Recall
that $K_{\lambda,\mu}$ is the number of tableaux of shape $\lambda$ and
${\mathcal{U}_{q}(\mathfrak{sl}_{\infty})}$-weight $\mu$.

\begin{Exa}
\label{exa_peri_tab}Take $e=2$, $\mathbf{s}=(2,3,6)$ and $\mathbf{t}=(0,0,1)$.
Then%
\[
\mathfrak{\tau}=\left(
\begin{array}
[c]{cccccc}
& 2 & 4 & 5 & 7 & 8\\
1 & 3 & 6 &  &  & \\
2 & 5 &  &  &  &
\end{array}
\right)
\]
is a tableau of shape $\mathbf{s}\setminus\mathbf{t}$ and weight $\mu
=\omega_{8}+\omega_{6}+\omega_{5}+\omega_{3}+\omega_{2}$.
\end{Exa}

The peeling procedure described in \S \ \ref{subsec-peeling} can be adapted to
the skew tableaux by successively removing their periods.\ For a skew tableau
$\tau$, denote by $\mathrm{w}(\tau)$ the word obtained by reading the entries
in the rows of $\tau$ from right to left and from top to bottom. When it
exists, the $e$-period of $\tau$ is the subword $u$ of $w(\tau)$ of the form
$u=u_{0}\cdots u_{e-1}$ where for any $k=0,\ldots,e-1$

\begin{itemize}
\item $u_{k}=M-k$ with $M$ the largest entry in $w(\tau),$

\item $u_{k}$ is the rightmost letter of $w(\tau)$ equal to $M-k.$
\end{itemize}

When $\tau$ is $e$-periodic, we write $\tau^{-}$ for the skew tableau obtained
by deleting its period. By condition on the rows and the columns of $\tau$,
$\tau^{-}$ is also a skew tableau. Its shape can be written on the form
$\mathbf{s}^{\prime}\setminus\mathbf{t}$ with $\mathbf{s}^{\prime}%
\in\mathcal{T}_{l,\infty}$.

\noindent More generally, given a skew tableau $\tau$ of shape $\mathbf{s}%
\setminus\mathbf{t}$, define the skew tableau $\tau^{{{}^{\circ}}}$ of shape
$\mathbf{s}^{\circ}\setminus\mathbf{t}$ as the result of the following peeling procedure:

\begin{itemize}
\item If $\tau$ is not periodic or $\tau=\tau_{\mathbf{s}\setminus\mathbf{t}}$
with $\mathbf{s\in}\mathcal{T}_{l,e}$, then $\tau^{\circ}=\tau$ and
$\mathbf{s}^{\circ}=\mathbf{s}$.

\item Otherwise, $\tau^{\circ}=(\tau^{\circ})^{\prime}$and $\mathbf{s}^{\circ
}=(\mathbf{s}^{\circ})^{\prime}$.
\end{itemize}

When $\tau^{\circ}=\uemptyset$ is the empty tableau, we have%
\[
\mathrm{wt}(\tau)_{\infty}=\sum_{T}\omega_{M(T)}%
\]
where $T$ runs over the $e$-periods of $\tau$ and for any period, $M(T)$ is
the largest integer in $T$. Write $\pi_{e}^{+}$ for the set of ${\mathcal{U}%
_{q}(\mathfrak{sl}_{\infty})}$-weights which are linear combinations of the
$\omega_{j},j\in\mathbb{Z}$ with nonnegative integer coefficients. When
$\tau^{\circ}=\uemptyset$, we have $\mathrm{wt}(\tau)_{\infty}\in\pi_{e}^{+}$.

\begin{Def}
\label{Def-tabloids}A \textit{totally periodic skew tableau} of shape
$\mathbf{s}\setminus\mathbf{t}$ is a skew tableau $\tau$ of shape
$\mathbf{s}\setminus\mathbf{t}$ such that

\begin{enumerate}
\item Each row $c=0,\ldots,l-1$ contains integers greater than $t_{c}.$

\item We have $\tau^{\circ}=\uemptyset$.
\end{enumerate}
\end{Def}

We denote by $\mathrm{Tab}_{\mathbf{s}\setminus\mathbf{t}}^{e}$ the set of
totally $e$-periodic skew tableaux of shape $\mathbf{s}\setminus\mathbf{t}$.
For any $\gamma\in\pi_{e}^{+}$, let $\mathrm{Tab}_{\mathbf{s}\setminus
\mathbf{t},\gamma}^{e}$ be the subset of $\mathrm{Tab}_{\mathbf{s}%
\setminus\mathbf{t}}^{e}$ of tableaux with ${\mathcal{U}_{q}(\mathfrak{sl}%
_{\infty})}$-weight $\gamma$.

\begin{Exa}
By applying the peeling procedure to the tableau $\tau$ of Example
\ref{exa_peri_tab}, we first obtain the sequence of tableaux
\[
\tau=\left(
\begin{array}
[c]{cccccc}
& 2 & 4 & 5 & 7 & 8\\
1 & 3 & 6 &  &  & \\
2 & 5 &  &  &  &
\end{array}
\right)  ,\tau^{(1)}=\left(
\begin{array}
[c]{cccc}
& 2 & 4 & 5\\
1 & 3 & 6 & \\
2 & 5 &  &
\end{array}
\right)  ,\tau^{(2)}=\left(
\begin{array}
[c]{cccc}
& 2 & 4 & 5\\
1 & 3 &  & \\
2 &  &  &
\end{array}
\right)  \text{ and }\tau^{(3)}=\left(
\begin{array}
[c]{cc}
& 2\\
1 & 3\\
2 &
\end{array}
\right)  .
\]
The tableau $\tau^{(3)}$ has shape $\mathbf{s}^{(3)}\setminus\mathbf{t}$ with
$\mathbf{t}=(0,0,1)$ and $\mathbf{s}^{(3)}=(1,2,2)$. Since $\tau^{(3)}\neq
\tau_{\mathbf{s}\setminus\mathbf{t}}$, the peeling procedure goes on.\ We
obtain%
\[
\tau^{(4)}=\left(
\begin{array}
[c]{cc}
& 2\\
1 & \\
&
\end{array}
\right)
\]
which has shape $\mathbf{s}^{(4)}\setminus\mathbf{t}$ with $\mathbf{s}%
^{(4)}=(0,1,2)$. Now $\mathbf{s}^{(4)}\notin\mathcal{T}_{l,e}$, so the
procedure finally yields $\tau^{(5)}=\tau^{\circ}=\uemptyset$. Therefore,
$\tau$ is totally $2$-periodic.
\end{Exa}

\subsection{Decomposition of $\mathcal{G}_{\mathbf{s},\infty}$%
\label{subsec_dec_Fock_infinite}}

In the sequel we assume $\mathbf{s\in}\mathcal{T}_{l,\infty}$ is fixed.\ By a
slight abuse of notation, we will identify each vertex $({\boldsymbol{\lambda
}},\mathbf{s})$ of $\mathcal{G}_{\mathbf{s},\infty}$ with its symbol
$\mathfrak{B}({\boldsymbol{\lambda}},\mathbf{s})$. For any $\mathbf{v\in
}\mathcal{T}_{l,\infty}$, let $\mathcal{H}_{\mathbf{s},\infty}^{\mathbf{v}}$
be the set of highest weight vertices in $\mathcal{G}_{\mathbf{s},\infty}$ of
highest weight $\Lambda_{\mathbf{v},\infty}$.

Consider $\mathfrak{B}({\boldsymbol{\lambda}},\mathbf{s})\in\mathcal{H}%
_{\mathbf{s},\infty}^{\mathbf{v}}$.\ For any fixed $k\in\mathbb{Z}$, the
contribution of all the integers $k$ in $\mathfrak{B}({\boldsymbol{\lambda}%
},\mathbf{s})$ to $\mathrm{wt}(\mathfrak{B}({\boldsymbol{\lambda}}%
,\mathbf{s}))_{\infty}$ is equal to $d_{k}\varepsilon_{k}$ where $d_{k}$ is
the number of occurrences of $k$ in $\mathfrak{B}({\boldsymbol{\lambda}%
},\mathbf{s})$. Each row contains at most a letter $k$, therefore $d_{k}\leq
l$ and $d_{k}=l$ if and only if $k$ appear in each row of $\mathfrak{B}%
({\boldsymbol{\lambda}},\mathbf{s})$. Since $\mathfrak{B}({\boldsymbol{\lambda
}},\mathbf{s})$ has weight $\Lambda_{\mathbf{v},\infty}$, we must have
$d_{k}=l$ for any $k\leq v_{0}$ and $d_{k}<l$ otherwise. This means that the
maximal integer $m$ such that $\mathfrak{B}({\boldsymbol{\lambda}}%
,\mathbf{s})$ contains each integer $k<m$ defined in Remark
\ref{Rem_HW_infinite} is equal to $v_{0}+1$. Let $\mathfrak{B}%
({\boldsymbol{\lambda}},\mathbf{s})_{\mathbf{v}}$ be the truncated symbol
obtained by deleting in $\mathfrak{B}({\boldsymbol{\lambda}},\mathbf{s})$ the
entries less or equal to $v_{0}$. By Remark \ref{Rem_HW_infinite} $(2)$, the
reading of $\mathfrak{B}({\boldsymbol{\lambda}},\mathbf{s})_{\mathbf{v}}$ is a
reverse lattice word.

\begin{Exa}
One verifies that
\[
\mathfrak{B}({\boldsymbol{\lambda}},\mathbf{s})=\left(
\begin{array}
[c]{cccccccccc}%
\ldots & -1 & 0 & 1 & 2 & 3 & 5 & 6 &  & \\
\ldots & -1 & 0 & 1 & 2 & 4 &  &  &  & \\
\ldots & -1 & 1 & 2 & 3 &  &  &  &  & \\
\ldots & -1 & 0 &  &  &  &  &  &  &
\end{array}
\right)
\]
with $\mathbf{s}=(0,2,3,5)$ is of highest weight $\Lambda_{\mathbf{v},\infty}$
with $\mathbf{v}=(-1,2,3,6)$. Then the reading of
\begin{equation}
\mathfrak{B}({\boldsymbol{\lambda}},\mathbf{s})_{\mathbf{v}}=\left(
\begin{array}
[c]{cccccccc}%
0 & 1 & 2 & 3 & 5 & 6 &  & \\
0 & 1 & 2 & 4 &  &  &  & \\
1 & 2 & 3 &  &  &  &  & \\
0 &  &  &  &  &  &  &
\end{array}
\right)  \label{Btrunc}%
\end{equation}
is the reverse lattice word
\[
w=65321042103210.
\]

\end{Exa}

Set $\mathbf{t}(\mathbf{v})=(v_{0},\ldots,v_{0})\in\mathbb{Z}^{l}$. Set
$\lambda=\mathbf{v}\setminus\mathbf{t}(\mathbf{v})$. Then $\lambda$ can be
regarded as an ordinary Young diagram. We define $\lambda^{\ast}$ has the
conjugate diagram of $\lambda$. We now associate to $\mathfrak{B}%
({\boldsymbol{\lambda}},\mathbf{s})_{\mathbf{v}}$ a tableau $T$ of shape
$\lambda(\mathbf{v})=\lambda^{\ast}$ and weight
\[
\mu(\mathbf{v})=\sum_{c=0}^{l-1}\mu_{c}\varepsilon_{c}%
\]
where for any $c=0,\ldots,l-1$, $\mu_{c}=s_{c}-v_{0}$ is the length of the
$c$-th row of $\mathfrak{B}({\boldsymbol{\lambda}},\mathbf{s})_{\mathbf{v}}$.
Observe that $\lambda^{\ast}$ is simply the sequence recording the number of
occurrences of each integer $k>v_{0}$ in $\mathfrak{B}({\boldsymbol{\lambda}%
},\mathbf{s})_{\mathbf{v}}$ (see the example below).\ Our procedure is a
variant of the one-to-one correspondence (reflecting the Schur duality)
described in \cite{NY} between the highest weight vertices of the
${\mathcal{U}_{q}(\mathfrak{sl}_{n})}$-Fock spaces and the semi-standard tableaux.

\noindent First normalize $\mathfrak{B}({\boldsymbol{\lambda}},\mathbf{s}%
)_{\mathbf{v}}$ by translating its entries by $-v_{0}$. Write $\mathfrak{B}%
({\boldsymbol{\lambda}},\mathbf{s})_{\mathbf{v}}^{t}$ for the resulting
truncated symbol. It has entries in $\mathbb{Z}_{>0}$ and its reading is a
reverse lattice word. Let $T^{(0)}$ be the tableau with one column containing
$\mu_{0}$ letters $1$. Assume the sequence of tableaux $T^{(0)},\ldots
,T^{(c-1)},$ $c<l-1$ is defined. Then $T^{(c)}$ is obtained by adding in
$T^{(c-1)}$ exactly $\mu_{c}$ letters $c+1$ at distance from the top row given
by the nonnegative integers appearing in the $c$-th row of $\mathfrak{B}%
({\boldsymbol{\lambda}},\mathbf{s})_{\mathbf{v}}^{t}$. Since the reading of
$\mathfrak{B}({\boldsymbol{\lambda}},\mathbf{s})_{\mathbf{v}}^{t}$ is a
reverse lattice word, $T^{(c)}$ is in fact a semi-standard tableau. We set
$T=T^{(l-1)}$.

\begin{Exa}
Let us compute $T$ for $\mathfrak{B}({\boldsymbol{\lambda}},\mathbf{s}%
)_{\mathbf{v}}$ as in (\ref{Btrunc}). We have $v_{0}=-1$%
\[
\mathfrak{B}({\boldsymbol{\lambda}},\mathbf{s})_{\mathbf{v}}^{t}=\left(
\begin{array}
[c]{cccccccc}%
1 & 2 & 3 & 4 & 6 & 7 &  & \\
1 & 2 & 3 & 5 &  &  &  & \\
2 & 3 & 4 &  &  &  &  & \\
1 &  &  &  &  &  &  &
\end{array}
\right)
\]
and we successively obtain for the tableaux $T^{(c)}$%
\[
T^{(0)}=\left(
\begin{array}
[c]{c}%
1
\end{array}
\right)  \text{, }T^{(1)}=\left(
\begin{array}
[c]{c}%
1\\
2\\
2\\
2
\end{array}
\right)  \text{, }T^{(2)}=\left(
\begin{array}
[c]{cc}%
1 & 3\\
2 & 3\\
2 & 3\\
2 & \\
3 &
\end{array}
\right)  \text{ and }T^{(3)}=\left(
\begin{array}
[c]{ccc}%
1 & 3 & 4\\
2 & 3 & 4\\
2 & 3 & 4\\
2 & 4 & \\
3 &  & \\
4 &  & \\
4 &  &
\end{array}
\right)  .
\]
We verify that $T^{(3)}=T$ has shape $\lambda(\mathbf{v})=(3,3,3,2,1,1,1)$ and
weight $\mu(\mathbf{v})=(1,3,4,6)$.
\end{Exa}

The previous procedure is reversible (for $\mathbf{s},\mathbf{v\in}%
\mathcal{T}_{l,\infty}$ fixed)$.\ $Starting from $T$ a tableau of shape
$\lambda(\mathbf{v})$ and weight $\mu(\mathbf{v})$, we can construct a
truncated symbol $\mathfrak{B}({\boldsymbol{\lambda}},\mathbf{s})_{\mathbf{v}%
}^{t},$ next $\mathfrak{B}({\boldsymbol{\lambda}},\mathbf{s})_{\mathbf{v}}$ by
translating the entries by $v_{0}$. This proves that the cardinality of
$\mathcal{H}_{\mathbf{s},\infty}^{\mathbf{v}}$ is finite and equal to
$K_{\lambda(\mathbf{v}),\mu(\mathbf{v})}$ the number of tableaux of shape
$\lambda(\mathbf{v})$ and weight $\mu(\mathbf{v})$.\ We thus obtain the
following theorem.

\begin{Th}
\label{Th_dec_Fock_inf}Consider $\mathbf{s\in}\mathcal{T}_{l,\infty}$. As a
${\mathcal{U}_{q}(\mathfrak{sl}_{\infty})}$-module, the Fock space
$\mathcal{F}_{\mathbf{s},\infty}$ decomposes as%
\[
\mathcal{F}_{\mathbf{s},\infty}=\bigoplus\limits_{\mathbf{v\in}\mathcal{T}%
_{l,\infty}}V_{\infty}(\mathbf{v})^{\oplus K_{\lambda(\mathbf{v}%
),\mu(\mathbf{v})}}.
\]

\end{Th}

\subsection{Branching rule for the restriction of $V_{\infty}(\mathbf{s})$ to
${\mathcal{U}}_{q}^{\prime}{(\widehat{\mathfrak{sl}_{e}})}$}

Consider $\mathbf{s\in}\mathcal{T}_{l,\infty}$. We now give the decomposition
of $\mathcal{G}_{\infty,\mathbf{s}}(\uemptyset)$ into its ${\mathcal{U}}%
_{q}^{\prime}{(\widehat{\mathfrak{sl}_{e}})}$-connected components. By
Corollary \ref{Cor_branching}, this reflects the branching rule for the
restriction of $V_{\infty}(\mathbf{s})$ from ${\mathcal{U}_{q}(\mathfrak{sl}%
_{\infty})}$ to the ${\mathcal{U}}_{q}^{\prime}{(\widehat{\mathfrak{sl}_{e}}%
)}$ action.\ By our assumption we have $s_{0}\leq\cdots\leq s_{l-1}$. It is
then easy to describe the symbols associated with the $l$-partitions appearing
in $\mathcal{G}_{\infty,\mathbf{s}}(\uemptyset)$. Indeed, $\mathfrak{B}%
({\boldsymbol{\lambda}},\mathbf{s})\in\mathcal{G}_{\infty,\mathbf{s}%
}(\uemptyset)$ if and only if it is semistandard (see \cite{JL}). This means
that its columns weakly increase from top to bottom.

Assume that $\mathfrak{B}({\boldsymbol{\lambda}},\mathbf{s})$ is totally
periodic in $\mathcal{G}_{\infty,\mathbf{s}}(\uemptyset)$. Set $\mathbf{s}%
^{\circ}=(s_{0}^{\circ},\ldots,s_{l-1}^{\circ})\in\mathcal{T}_{l,e}$. We
define the level $l$-part of the symbol $\mathfrak{B}({\boldsymbol{\lambda}%
},\mathbf{s})$ as the symbol $\mathfrak{B}({\boldsymbol{\lambda}}%
,\mathbf{s})_{l}=\mathfrak{B}(\uemptyset,\mathbf{s}^{\circ})$ which can be
regarded as a subsymbol of $\mathfrak{B}({\boldsymbol{\lambda}},\mathbf{s})$
in a natural sense. The level $0$-part of $\mathfrak{B}({\boldsymbol{\lambda}%
},\mathbf{s})$ is then%
\[
\mathfrak{B}({\boldsymbol{\lambda}},\mathbf{s})_{0}=\mathfrak{B}%
({\boldsymbol{\lambda}},\mathbf{s})\setminus\mathfrak{B}(\uemptyset,\mathbf{s}%
^{\circ}).
\]
For any $\mathbf{t\in}\mathcal{T}_{l,e}$, we set%
\[
S_{\mathbf{t}}=\{\mathfrak{B}({\boldsymbol{\lambda}},\mathbf{s})\in
\mathcal{G}_{\infty,\mathbf{s}}(\uemptyset)\text{ totally periodic}%
\mid\mathbf{s}^{\circ}=\mathbf{t}\}.
\]
The following lemma is immediate from the definitions of the peeling
procedures on symbols and tableaux.

\begin{lemma}
\label{lem_Util}Fix $\mathbf{t\in}\mathcal{T}_{l,e}.$ The map $\psi
:\mathfrak{B}({\boldsymbol{\lambda}},\mathbf{s})\longmapsto\mathfrak{B}%
({\boldsymbol{\lambda}},\mathbf{s})_{0}$ is a one-to-one correspondence
between the sets $S_{\mathbf{t}}$ and $\mathrm{Tab}_{\mathbf{s}\setminus
\mathbf{t}}^{e}$. We have moreover%
\begin{equation}
\mathrm{wt}({\boldsymbol{\lambda}},\mathbf{s})_{\infty}=\Lambda_{\mathbf{t}%
,\infty}+\mathrm{wt}(\mathfrak{B}({\boldsymbol{\lambda}},\mathbf{s})_{0}).
\label{deco}%
\end{equation}

\end{lemma}

\begin{Exa}
Take $e=2$, $\mathbf{s}=(2,3,6)$ and%
\[
\mathfrak{B}({\boldsymbol{\lambda}},\mathbf{s})=\left(
\begin{array}
[c]{ccccccccc}%
-2 & -1 & 0 & 1 & 2 & 4 & 5 & 7 & 8\\
-2 & -1 & 0 & 1 & 3 & 6 &  &  & \\
-2 & -1 & 0 & 2 & 5 &  &  &  &
\end{array}
\right)  .
\]
We obtain
\[
\mathfrak{B}({\boldsymbol{\lambda}},\mathbf{s})_{0}=\left(
\begin{array}
[c]{cccccc}
& 2 & 4 & 5 & 7 & 8\\
1 & 3 & 6 &  &  & \\
2 & 5 &  &  &  &
\end{array}
\right)  \in\mathrm{Tab}_{\mathbf{s}\setminus\mathbf{t}}^{2}%
\]
with $\mathbf{t}=(0,0,1)$. We have $\mathrm{wt}({\boldsymbol{\lambda}%
},\mathbf{s})_{\infty}=\Lambda_{\mathbf{t},\infty}+\omega_{8}+\omega
_{6}+\omega_{5}+\omega_{3}+\omega_{2}$.
\end{Exa}

Let $P_{e,\infty}$be the subset of $P_{\infty}$ of weights $\nu$ which can be
written on the form
\begin{equation}
\nu=\Lambda_{\mathbf{t}(\nu),\infty}+\gamma(\nu)\text{ with }\mathbf{t}%
(\nu)\mathbf{\in}\mathcal{T}_{l,e}\text{ and }\gamma(\nu)=\sum_{k>t_{0}%
(\nu)+e}a_{k}\omega_{k}\in\pi_{e}^{+}\label{dec(nu)}%
\end{equation}
where all but a finite number of the coefficients $a_{k}$ are equal to $0$.
Observe that the previous decomposition is then unique. Indeed, for any
$\mathbf{t\in}\mathcal{T}_{l,e}$ and any $k>t_{0}+e$, the weight
$\Lambda_{\mathbf{t},\infty}+\omega_{k}$ cannot be written on the form
$\Lambda_{\mathbf{t}^{\prime},\infty}$ with $\mathbf{t}^{\prime}\in
\mathcal{T}_{l,e}$. Let $\mathfrak{B}({\boldsymbol{\lambda}},\mathbf{s})$ a
highest weight vertex of $\mathcal{G}_{\mathbf{s},e}$ with weight $\nu$.

\begin{lemma}
\label{Lem_dec_can}The ${\mathcal{U}_{q}(\mathfrak{sl}_{\infty})}$-weight of
$\mathfrak{B}({\boldsymbol{\lambda}},\mathbf{s})$ belongs to $P_{e,\infty}$.
Moreover, we have%
\[
\mathbf{t}(\nu)=\mathbf{s}^{\circ}\text{ and }\gamma(\nu)=\mathrm{wt}%
(\mathfrak{B}({\boldsymbol{\lambda}},\mathbf{s})_{0})
\]
where $\mathbf{s}^{\circ}$ and $\mathfrak{B}({\boldsymbol{\lambda}}%
,\mathbf{s})_{0}$ are obtained by the peeling procedure as in (\ref{deco}).
\end{lemma}

\begin{proof}
In view to (\ref{deco}), the weight $\nu$ decomposes on the form
\[
\nu=\Lambda_{\mathbf{s}^{\circ},\infty}+\mathrm{wt}(\mathfrak{B}%
({\boldsymbol{\lambda}},\mathbf{s})_{0}%
\]
where by Theorem \ref{Th_HW} and Lemma \ref{lem_Util}, we have $\mathbf{s}%
^{\circ}\mathbf{\in}\mathcal{T}_{l,e}$ and $\mathrm{wt}(\mathfrak{B}%
({\boldsymbol{\lambda}},\mathbf{s})_{0}\in\pi_{e}^{+}$. Set $\mathrm{wt}%
(\mathfrak{B}({\boldsymbol{\lambda}},\mathbf{s})_{0}=\sum_{k\in\mathbb{Z}%
}a_{k}\omega_{k}$.\ The entries of $\mathfrak{B}({\boldsymbol{\lambda}%
},\mathbf{s})_{0}$ are those of the periods of $\mathfrak{B}%
({\boldsymbol{\lambda}},\mathbf{s})$ and $a_{k}$ is the number of periods
$\{k,\ldots,k-e+1\}$ in $\mathfrak{B}({\boldsymbol{\lambda}},\mathbf{s})_{0}$.
Let $k_{0}$ be the minimal integer such that $a_{k_{0}}\neq0$.  By definition
of the peeling procedure, the addition of the letters $\{k_{0}-e+1,\ldots
,k_{0}\}$ in the symbol $\mathfrak{B}(\uemptyset,\mathbf{s}^{\circ})$, yields
a symbol $\mathfrak{B}(\uemptyset,\mathbf{u})$ with $\mathbf{u\in}%
\mathcal{T}_{l,\infty}$ but $\mathbf{u\notin}\mathcal{T}_{l,e}$. Since
$\mathbf{u\in}\mathcal{T}_{l,\infty}$, we must have $k_{0}-e+1>s_{0}^{\circ}$,
that is $k_{0}\geq s_{0}^{\circ}+e$. We cannot have $k_{0}=s_{0}^{\circ}+e$,
otherwise $\mathbf{u}=(s_{1}^{\circ},\ldots,s_{l}^{\circ},s_{0}^{\circ}%
+e)\in\mathcal{T}_{l,e}$. Thus $k_{0}>s_{0}^{\circ}+e$.\ Since the
decomposition (\ref{dec(nu)}) is unique, this imposes that $\mathbf{t}%
(\nu)=\mathbf{s}^{\circ}$ and $\gamma(\nu)=\mathrm{wt}(\mathfrak{B}%
({\boldsymbol{\lambda}},\mathbf{s})_{0})$ as desired.
\end{proof}

\begin{Prop}
Consider a totally periodic symbol $\mathfrak{B}({\boldsymbol{\lambda}%
},\mathbf{s})$ in $\mathcal{G}_{\infty,\mathbf{s}}(\uemptyset)$ of
${\mathcal{U}_{q}(\mathfrak{sl}_{\infty})}$-weight $\nu$.

\begin{enumerate}
\item The successive symbols appearing during the peeling procedure of
$\mathfrak{B}({\boldsymbol{\lambda}},\mathbf{s})$ of $\mathcal{G}%
_{\infty,\mathbf{s}}(\uemptyset)$ remain semistandard.

\item The number of highest weight vertices in $\mathcal{G}_{\infty
,\mathbf{s}}(\uemptyset)$ with ${\mathcal{U}_{q}(\mathfrak{sl}_{\infty})}%
$-highest weight $\nu\in P_{e,\infty}$ is finite equal to $m_{\mathbf{s},\nu
}^{e}=\left\vert \mathrm{Tab}_{\mathbf{s}\setminus\mathbf{t}(\nu),\gamma(\nu
)}\right\vert .$
\end{enumerate}
\end{Prop}

\begin{proof}
Assertion 1 follows from the fact that the columns of $\mathfrak{B}%
({\boldsymbol{\lambda}},\mathbf{s})$ increase from top to bottom and each
entry $k$ in a period is the lowest possible occurrence of the integer $k$ in
the symbol considered. Consider $\mathfrak{B}({\boldsymbol{\lambda}%
},\mathbf{s})$ of highest weight $\nu$. By Lemma \ref{Lem_dec_can}, we have
the decomposition $\nu=\mathbf{s}^{\circ}+\mathrm{wt}(\mathfrak{B}%
({\boldsymbol{\lambda}},\mathbf{s})_{0})$. Then the restriction of the
bijection $\psi$ defined in Lemma \ref{lem_Util} to the symbols of weight
$\nu$ yields a one-to-one correspondence between the symbols $\mathfrak{B}%
({\boldsymbol{\lambda}},\mathbf{s})$ of highest weight $\nu$ and the tableaux
$\mathfrak{B}({\boldsymbol{\lambda}},\mathbf{s})_{0}$ of shape $\mathbf{s}%
\setminus\mathbf{s}^{\circ}$ and weight $\gamma(\nu)$. Assertion 2 follows.
\end{proof}

\bigskip

We thus obtain the following theorem.

\begin{Th}
Assume $e$ is finite and consider $\mathbf{s\in}\mathcal{T}_{l,\infty}$.
\label{Dec_Fock}

\begin{enumerate}
\item The crystal $\mathcal{G}_{\infty,\mathbf{s}}(\uemptyset)$ decomposes
into irreducible ${\mathcal{U}}_{q}^{\prime}{(\widehat{\mathfrak{sl}_{e}})}%
$-components whose highest weight vertices are also weight vertices for the
${\mathcal{U}_{q}(\mathfrak{sl}_{\infty})}$-structure.

\item The ${\mathcal{U}_{q}(\mathfrak{sl}_{\infty})}$-weight of such a vertex
belongs to $P_{e,\infty}$

\item The number of highest weight vertices in $\mathcal{G}_{\infty
,\mathbf{s}}(\uemptyset)$ with ${\mathcal{U}_{q}(\mathfrak{sl}_{\infty})}%
$-highest weight $\nu\in P_{e,\infty}$ is finite equal to the cardinality
$m_{\mathbf{s},\nu}^{e}=\left\vert \mathrm{Tab}_{\mathbf{s}\setminus
\mathbf{t}(\nu),\gamma(\nu)}\right\vert .$
\end{enumerate}
\end{Th}

By combining with Theorem \ref{Th_dec_Fock_inf}, this yields the decomposition
of the Fock space in its irreducible ${\mathcal{U}}_{q}^{\prime}%
{(\widehat{\mathfrak{sl}_{e}})}$-components.

\begin{Th}
Assume $e$ is finite and consider $\mathbf{s\in}\mathcal{T}_{l,\infty}$.

\begin{enumerate}
\item The crystal $\mathcal{G}_{\mathbf{s},e}$ decomposes into irreducible
${\mathcal{U}}_{q}^{\prime}{(\widehat{\mathfrak{sl}_{e}})}$-components whose
highest weight vertices are also weight vertices for the ${\mathcal{U}%
_{q}(\mathfrak{sl}_{\infty})}$-structure $\mathcal{G}_{\mathbf{s},\infty}$.

\item The ${\mathcal{U}_{q}(\mathfrak{sl}_{\infty})}$-weight of such a vertex
belongs to $P_{e,\infty}$

\item The number of ${\mathcal{U}}_{q}^{\prime}{(\widehat{\mathfrak{sl}_{e}}%
)}$-highest weight vertices in $\mathcal{G}_{e}$ with ${\mathcal{U}%
_{q}(\mathfrak{sl}_{\infty})}$-highest weight $\nu\in P_{e,\infty}$ is finite
equal to $M_{\mathbf{s},\nu}^{e}=\sum_{\mathbf{v\in}\mathcal{T}_{l,\infty}%
}K_{\lambda(\mathbf{v}),\mu(\mathbf{v})}m_{\mathbf{v},\nu}^{e}$.
\end{enumerate}
\end{Th}

\end{document}